\documentclass[11pt]{amsart}
\usepackage{amsmath,amssymb,amsthm}
\usepackage[margin=1.15in]{geometry}
\usepackage[colorlinks=true,linkcolor=blue,citecolor=blue,urlcolor=blue]{hyperref}

\newtheorem{theorem}{Theorem}[section]
\newtheorem{lemma}[theorem]{Lemma}
\newtheorem{corollary}[theorem]{Corollary}
\theoremstyle{definition}
\newtheorem{definition}[theorem]{Definition}
\newtheorem{example}[theorem]{Example}
\theoremstyle{remark}
\newtheorem{remark}[theorem]{Remark}

\numberwithin{equation}{section}
\newcommand{\N}{\mathbb{N}}
\newcommand{\R}{\mathbb{R}}
\newcommand{\C}{\mathbb{C}}
\newcommand{\Z}{\mathbb{Z}}
\newcommand{\cP}{\mathcal{P}}
\newcommand{\cD}{\mathcal{D}}
\newcommand{\cA}{\mathcal{A}}
\newcommand{\sm}{\operatorname{sm}}
\renewcommand{\lg}{\operatorname{lg}}
\newcommand{\muP}{\mu_{\mathcal{P}}}
\newcommand{\eps}{\varepsilon}
\newcommand{\dd}{\,\mathrm{d}}
\newcommand{\norm}[1]{\lVert #1\rVert_\infty}

\begin{document}

\title[Partition-theoretic density formulas]{Partition-theoretic density formulas: radial Tauberian theory and non-radial instability}

\author{Mohammad F. Marashdeh}
\address{Department of Mathematics, Mutah University, Karak, Jordan}
\email{marashdeh@mutah.edu.jo}


\begin{abstract}
Ono, Schneider and Wagner found a partition-theoretic analogue of Alladi's duality formula: for many sets $S\subseteq\N$, the series $-\sum\mu_{\mathcal P}(\lambda)q^{|\lambda|}$ over partitions into distinct parts with smallest part in $S$ tends to the density of $S$ as $q\to1$. We determine when this happens. Radially, the summation kernel is a Gumbel density, and Wiener's theorem shows that for bounded $f$ the associated series tends to $L$ if and only if $f$ has mean $L$ on every interval $(x,x+cx/\log x]$; natural density does not suffice. Multiplied by $\log(1/(1-q))$, the series tends to $\delta$ if and only if $S$ has $(c\delta+o(1))x/(\log x)^2$ elements in each such interval; this gives a partition analogue of Alladi's theorem for Chebotarev sets of primes. Off the real axis the largest weights grow like $e^{\Psi_\beta/t}$ with $\Psi_\beta>0$. Along $q=e^{-t(1+i\beta)}$, the formula for residue classes modulo $m\ge3$ holds exactly when $|\beta|<\beta_m$, an explicit Clausen-function threshold with $\beta_m\to0$, and it fails almost surely for random sets of bounded discrepancy. These results correct several published claims.
\end{abstract}

\maketitle

\section{Introduction}

For an integer $n\ge2$ let $p_{\min}(n)$ denote its least prime factor. Alladi~\cite{Alladi} proved that for coprime integers $\ell$ and $k\ge1$,
\begin{equation}\label{eq:alladi}
-\sum_{\substack{n\ge2\\ p_{\min}(n)\equiv \ell \ (\mathrm{mod}\ k)}}\frac{\mu(n)}{n}=\frac{1}{\varphi(k)} .
\end{equation}
Dawsey~\cite{Dawsey} replaced the congruence condition by a Chebotarev condition, Sweeting and Woo~\cite{SweetingWoo} extended the result to Galois extensions of number fields, and Kural, McDonald and Sah~\cite{KMS} showed that for any set $\mathcal S$ of prime ideals of a number field possessing a natural density $\delta(\mathcal S)$, the analogous sum of $-\mu(\mathfrak a)/N(\mathfrak a)$ over ideals whose least prime factor lies in $\mathcal S$ equals $\delta(\mathcal S)$. In this multiplicative setting, the existence of natural density is sufficient.

Ono, Schneider and Wagner~\cite{OSW1,OSW2} introduced a partition-theoretic counterpart of~\eqref{eq:alladi}. Let $\cP$ be the set of integer partitions (including the empty partition $\emptyset$), and $\cD\subset\cP$ the partitions into distinct parts. For $\lambda\neq\emptyset$ write $|\lambda|$, $\ell(\lambda)$, $\sm(\lambda)$ and $\lg(\lambda)$ for its size, number of parts, smallest part and largest part, and put $\muP(\lambda)=(-1)^{\ell(\lambda)}$ if $\lambda\in\cD$ and $\muP(\lambda)=0$ otherwise. We put $\sm(\emptyset)=\lg(\emptyset)=0$ and, for an arithmetic function $f$, $f(0)=0$, so that the empty partition contributes nothing to the sums below. For $S\subseteq\N=\{1,2,3,\dots\}$ and $|q|<1$ one has (Lemma~\ref{lem:identities})
\begin{equation}
F_S(q):=-\sum_{\substack{\lambda\in\cD\setminus\{\emptyset\}\\ \sm(\lambda)\in S}}\muP(\lambda)q^{|\lambda|}=\sum_{n\in S}q^n\,(q^{n+1};q)_\infty ,
\end{equation}
and the question is whether
\begin{equation}\label{eq:OSWlimit}
\lim_{q\to1^-}F_S(q)=d_S:=\lim_{x\to\infty}\frac{\#\{n\le x:\ n\in S\}}{x}.
\end{equation}
In~\cite{OSW1} this is proved for arithmetic progressions and for the integers not divisible by $p^k$ for any prime $p$ below a fixed bound, from which the value $1/\zeta(k)$ for the $k$-free integers is obtained by letting the bound tend to infinity; limits at roots of unity are treated as well, and all limits there are taken as $q$ tends to $1$ (or to a root of unity) from within the unit disk. In~\cite{OSW2} a general theory is developed for sets called \emph{$q$-commensurate}, defined by the requirement that the relevant asymptotic arguments remain valid. For the Abel-type $q$-density $(1-q)\sum_{n\in S}q^n$, the limit $d_S$ holds for every $S$ possessing a natural density, by the theorem of Frobenius recalled below. On this basis~\cite[Remark~1.8]{OSW2} states that the formulas of that paper, among them~\eqref{eq:OSWlimit}, ``hold for every subset $S\subseteq\N$ if one lets $q\to1^-$ strictly radially''. Schneider~\cite[Theorem~1.1 and Corollaries~1.2 and~1.3]{Schneider} (numbering as in arXiv:2011.08386v5) states that if an arithmetic function $f$ has mean value
\[
M(f):=\lim_{N\to\infty}\frac1N\sum_{n\le N}f(n)=L,
\]
then each of
\[
(q;q)_\infty\sum_{\lambda\in\cP}f(\lg(\lambda))q^{|\lambda|},\qquad -\sum_{\lambda\in\cP}\muP(\lambda)f(\sm(\lambda))q^{|\lambda|},\qquad (q;q)_\infty\sum_{n\ge1}\frac{f(n)q^n}{(q;q)_n}
\]
tends to $L$ as $q\to1$ in any Stolz sector $|1-q|\le M_0(1-|q|)$; no hypothesis on $f$ other than the existence of $M(f)$ is imposed. By Lemma~\ref{lem:identities} these three series coincide with
\begin{equation}\label{eq:Af}
\cA_f(q):=\sum_{n\ge1}f(n)\,q^n(q^{n+1};q)_\infty ,
\end{equation}
and $F_S=\cA_{\mathbf 1_S}$. The model for these statements is Frobenius' theorem~\cite{Frobenius}: if $M(f)=L$ then $(1-q)\sum_{n\ge1}f(n)q^n\to L$ as $q\to1^-$.

The kernels in the two situations are, however, of different shape. The Abel weights $(1-q)q^n$ with $q=e^{-t}$ are spread over an interval of length $\asymp 1/t$ beginning at the origin, so their relative width is of order one and Cesàro convergence controls them. The weights $\omega_n(q)=q^n(q^{n+1};q)_\infty$ in~\eqref{eq:Af} are the probabilities that a partition drawn from the Boltzmann measure $(q;q)_\infty q^{|\lambda|}$ has largest part $n$. They are concentrated in a window of length $\asymp 1/t$ around $x_t=t^{-1}\log(1/t)$, whose relative width is $\asymp 1/\log x_t$, and after centering and scaling they are given by the Gumbel density $e^{-u-e^{-u}}$. This is the Boltzmann form of the limit law of Erd\H{o}s and Lehner~\cite{ErdosLehner} for the largest part of a random partition (see also Fristedt~\cite{Fristedt}). Lemma~\ref{lem:gumbel} gives a total-variation version of this law with error $O(t)$. Since the kernel sees only intervals of length comparable to $x/\log x$, one should expect~\eqref{eq:OSWlimit} to depend on the distribution of $S$ in such intervals rather than on its natural density.

Our main result, Theorem~\ref{thm:main}, confirms this. For bounded $f$, $\cA_f(q)\to L$ as $q\to1^-$ if and only if
\[
\sum_{x<n\le x+cx/\log x}f(n)=\frac{cLx}{\log x}+o\Big(\frac{x}{\log x}\Big)\qquad(x\to\infty)
\]
for every fixed $c>0$; a single value of $c$ does not suffice (Remark~\ref{rem:onec}). The change of variable $v=\tfrac12(\log x)^2$ makes the kernel asymptotically translation invariant, with limiting kernel the Gumbel density. The Tauberian direction then follows from Wiener's theorem, because the Fourier transform of the Gumbel density is $\Gamma(1+i\xi)$, which has no real zeros. Theorem~\ref{thm:abelian} gives a quantitative Abelian bound: if $|\sum_{n\le x}f(n)-dx|\le\eps x/\log x$ for all large $x$, then $\limsup_{q\to1^-}|\cA_f(q)-d|\le 2\eps/e$. Theorem~\ref{thm:counter} shows that Theorem~\ref{thm:abelian} and Corollary~\ref{cor:consequences}(a) are sharp. An error term of exact order $x/\log x$, with arbitrarily small constant, need not give convergence. For every function exceeding $x/\log x$ by a factor tending to infinity, there are sets whose counting function stays within that function of $x/2$ and for which $F_S(q)$ oscillates between $0$ and $1$. Consequently (Corollary~\ref{cor:ND}), there is a set $S$ with natural density $\tfrac12$ for which $\liminf_{q\to1^-}F_S(q)=0$ and $\limsup_{q\to1^-}F_S(q)=1$. Thus the assertion of~\cite[Remark~1.8]{OSW2}, as it applies to~\eqref{eq:OSWlimit}, and Theorem~1.1 and Corollaries~1.2 and~1.3 of~\cite{Schneider} do not hold as stated, even for indicator functions and radial approach. The function $\mathbf 1_S$ is bounded, takes only the values $0$ and $1$, and has a mean value, and the radial path lies in every Stolz sector; thus it meets every hypothesis of these statements. The results of~\cite{OSW1,OSW2} for periodic sets and for $k$-free integers are unaffected; they are recovered in Example~\ref{ex:examples}.

The Gumbel description also yields a partition analogue of Alladi's theorem itself, which concerns primes. By Theorem~\ref{thm:primes}, for every $S\subseteq\N$ and $\alpha\ge0$ the limit $(\log(1/(1-q)))^{\alpha}F_S(q)\to\delta$ holds if and only if $S$ has $(c\delta+o(1))x/(\log x)^{1+\alpha}$ elements in every interval $(x,x+cx/\log x]$. Here the Tauberian condition is derived from positivity rather than assumed. For $\alpha=1$ and $S=\mathbb P$, a partition-theoretic prime number theorem is thereby equivalent to the equidistribution of primes in all intervals of length comparable to $x/\log x$; both statements hold. With the Chebotarev density theorem, the series over partitions into distinct parts whose smallest part is a prime $\equiv\ell\ (\mathrm{mod}\ k)$, with $(\ell,k)=1$, multiplied by $\log(1/(1-q))$, tends to $1/\varphi(k)$ (Corollary~\ref{cor:chebotarev}). The formulas of~\cite{OSW1,OSW2} give the value $0$ for every set of primes.

Off the real axis the situation is different. Along $q_t(\beta):=e^{-t(1+i\beta)}$ with $\beta\ne0$ the weights $\omega_n(q)$ reach size $\exp(\Psi_\beta/t)$ with $\Psi_\beta>0$ (Theorem~\ref{thm:instability}), so any limit must come from cancellation. For periodic $f$, Euler's identity and the modular behavior of $(q;q)_\infty$ reduce the question to the values of $\mathrm{Li}_2$ at roots of unity. For $m\ge3$ the formula holds for all residue classes modulo $m$ exactly when $|\beta|<\beta_m$, where $\beta_m$ is the minimum over $\vartheta=2\pi j/m$, $1\le j<m/2$, of the ratio of $\vartheta(2\pi-\vartheta)/4$ to the Clausen function $\mathrm{Cl}_2(\vartheta)$ (Theorem~\ref{thm:periodic} and Corollary~\ref{cor:AP}). For example $\beta_4=3\pi^2/(16G)$ with $G$ Catalan's constant, and $\beta_m\to0$. For a random set of bounded discrepancy, obtained by choosing one element from each pair $\{2k-1,2k\}$, the series is almost surely unbounded along $q_t(\beta)$ for every $\beta$ in a prescribed countable subset of $\R\setminus\{0\}$, although its radial limit exists. The density formula is therefore a radial phenomenon. Corollary~\ref{cor:stolz} records the consequences for~\cite{OSW1} and~\cite{Schneider}.

Theorem~\ref{thm:main} belongs to a classical family of Tauberian equivalences. For bounded sequences, summability by the methods of Borel and Valiron is equivalent to convergence of moving averages over windows of length $\asymp\sqrt x$, and the proofs rest on Wiener's theorem with a Gaussian kernel; see Bingham~\cite{Bingham84,Bingham89} and, for moving averages over general self-neglecting windows, Bingham and Goldie~\cite{BinghamGoldie}. The partition-theoretic series is a method of the same kind, with window $x/\log x$ and the Gumbel density in place of the Gaussian. Its kernel is not of convolution type in any fixed coordinate, and Lemma~\ref{lem:reduction} carries out the reduction to Wiener's theorem asymptotically. To our knowledge, no characterization of the sets satisfying~\eqref{eq:OSWlimit} and no example of a set with natural density for which~\eqref{eq:OSWlimit} fails have appeared before. In the multiplicative setting of~\cite{KMS}, ordering by norm produces averages controlled by natural density; the partition generating function produces instead a moving average over windows of relative length $1/\log x$.

\section{Preliminaries}

\subsection{Notation}
Throughout, $f:\N\to\C$ is bounded, $\norm{f}=\sup_n|f(n)|$, and for $S\subseteq\N$ we write $A_S(x)=\#\{n\le x:\ n\in S\}$. For $|q|<1$ and $n\ge1$ put $\omega_n(q)=q^n(q^{n+1};q)_\infty$. For $0<q<1$ we write $q=e^{-t}$ with $t>0$ and put
\[
L_t=\log(1/t)\quad(t>0),\qquad x_t=\frac{L_t}{t}\quad(0<t<1/e).
\]
The quantity $L_t$ is used for all $t>0$, while $x_t$, and $w(t)$ below, are used only for $0<t<1/e$, where $L_t>1$.
The symbol $o(1)$ refers to the limit indicated, with implied rates depending at most on $f$ and on fixed parameters named in the statement.

\subsection{Partition identities}

\begin{lemma}\label{lem:identities}
For $|q|<1$ and bounded $f$,
\begin{equation}
\begin{aligned}
\cA_f(q)&=(q;q)_\infty\sum_{\lambda\in\cP}f(\lg(\lambda))q^{|\lambda|}=-\sum_{\lambda\in\cD\setminus\{\emptyset\}}\muP(\lambda)f(\sm(\lambda))q^{|\lambda|}\\
&=(q;q)_\infty\sum_{n\ge1}\frac{f(n)q^n}{(q;q)_n},
\end{aligned}
\end{equation}
Moreover, for $|q|<1$,
\begin{equation}\label{eq:mass}
\sum_{n\ge1}\omega_n(q)=1-(q;q)_\infty ,
\end{equation}
and $\omega_n(q)>0$ for $0<q<1$.
\end{lemma}

\begin{proof}
The partitions with largest part $n$ have generating function $q^n/(q;q)_n$. Since $(q;q)_\infty/(q;q)_n=(q^{n+1};q)_\infty$, this gives the first and third expressions. A partition into distinct parts with smallest part $n$ consists of the part $n$ together with a set of distinct parts exceeding $n$, so
\[
\sum_{\lambda\in\cD,\ \sm(\lambda)=n}(-1)^{\ell(\lambda)}q^{|\lambda|}=-q^n\prod_{m>n}(1-q^m),
\]
which gives the second. All series converge absolutely because $f$ is bounded. For~\eqref{eq:mass}, put $P_n=(q^{n+1};q)_\infty$; then $P_n-P_{n-1}=q^nP_n=\omega_n(q)$, so $\sum_{n=1}^N\omega_n(q)=P_N-(q;q)_\infty$, and $P_N\to1$.
\end{proof}

Since $0<(q;q)_\infty\le 1-q$, the weights $\omega_n(q)$ form a subprobability vector of total mass tending to $1$.

\subsection{The Gumbel density}
Let
\[
g(u)=\exp(-u-e^{-u}),\qquad \mathcal G(u)=\exp(-e^{-u})=\int_{-\infty}^u g ,\qquad u\in\R .
\]
Then $g>0$, $\int_\R g=1$, and $g$ increases on $(-\infty,0]$ and decreases on $[0,\infty)$ with maximum $g(0)=e^{-1}$, so
\begin{equation}\label{eq:gprime}
\int_\R|g'(u)|\dd u=2g(0)=\frac2e .
\end{equation}
The substitution $s=e^{-u}$ gives, for $\xi\in\R$,
\begin{equation}\label{eq:fourier}
\int_\R g(u)e^{-i\xi u}\dd u=\int_0^\infty s^{i\xi}e^{-s}\dd s=\Gamma(1+i\xi)\neq0,
\end{equation}
since $\Gamma$ has no zeros, and $|\Gamma(1+i\xi)|^2=\pi\xi/\sinh(\pi\xi)$ for $\xi\ne0$.

\subsection{Tauberian and Abelian tools}

We use Wiener's Tauberian theorem in the following form~\cite{Wiener}; see~\cite[Chapter~9]{Rudin} or~\cite[Chapter~II]{Korevaar}.

\begin{theorem}[Wiener]\label{thm:wiener}
Let $\Phi\in L^\infty(\R)$ and $K\in L^1(\R)$ with $\widehat K(\xi)=\int_\R K(z)e^{-i\xi z}\dd z\ne0$ for all $\xi\in\R$. If $(K*\Phi)(w)\to a\widehat K(0)$ as $|w|\to\infty$, then $(k*\Phi)(w)\to a\widehat k(0)$ as $|w|\to\infty$ for every $k\in L^1(\R)$.
\end{theorem}

The Abelian counterpart of Theorem~\ref{thm:wiener} requires no condition on the kernel.

\begin{lemma}\label{lem:abelstep}
Let $\Theta\in L^\infty([0,\infty))$ satisfy $\int_w^{w+c}\Theta\to0$ as $w\to\infty$ for every $c>0$. Then $\int_0^\infty\Theta(v)k(v-w)\dd v\to0$ as $w\to\infty$ for every $k\in L^1(\R)$.
\end{lemma}

\begin{proof}
For $k=\mathbf 1_{[a,b]}$ and $w>-a$ the integral equals $\int_{w+a}^{w+b}\Theta\to0$. By linearity the claim holds for step functions. For general $k$ and a step function $k_0$, the two integrals differ by at most $\|\Theta\|_\infty\|k-k_0\|_{L^1}$, and step functions are dense in $L^1(\R)$.
\end{proof}

\begin{lemma}[Scheff\'e]\label{lem:scheffe}
Let $h_j\ge0$ and $h\ge 0$ be integrable on $\R$ with $h_j\to h$ pointwise and $\int h_j\to\int h$. Then $\int|h_j-h|\to0$.
\end{lemma}

\begin{proof}
$\int|h_j-h|=2\int(h-h_j)^++\int h_j-\int h$, and $\int(h-h_j)^+\to0$ by dominated convergence, since $0\le(h-h_j)^+\le h$.
\end{proof}

\subsection{Logarithmic-quadratic coordinates}
For $x\ge1$ put $v(x)=\tfrac12(\log x)^2$, with inverse $\chi(v)=e^{\sqrt{2v}}$ $(v\ge0)$, so that
\[
v'(x)=\frac{\log x}{x},\qquad \chi'(v)=\frac{\chi(v)}{\log \chi(v)} .
\]
For bounded $f$ define
\[
\psi_f(x)=f(\lceil x\rceil)\ \ (x>0),\qquad \Phi_f(v)=\psi_f(\chi(v))\ \ (v\ge0),
\]
\[
J_f(w)=\int_0^\infty\Phi_f(v)\,g(v-w)\dd v\qquad (w\in\R),
\]
and for $0<t<1/e$ put $w(t)=v(x_t)=\tfrac12(\log x_t)^2$. Since $x_t=L_te^{L_t}$ and $L_t$ decreases from $\infty$ to $1$ as $t$ increases from $0$ to $1/e$, the maps $t\mapsto x_t$ and $t\mapsto w(t)$ are continuous and strictly decreasing, with $x_t\to\infty$ and $w(t)\to\infty$ as $t\to0^+$. Substituting $v=v(x)$ gives
\begin{equation}\label{eq:Jx}
J_f(w)=\int_1^\infty\psi_f(x)\,v'(x)\,g\big(v(x)-w\big)\dd x .
\end{equation}
Since $\Phi_f$ is piecewise constant, $\int_{w}^{w+c}\Phi_f$ is defined for all $w\ge0$ and $c>0$.

\subsection{Dilogarithm, Clausen function and \texorpdfstring{$q$}{q}-products}\label{sec:dilog}
For $|z|\le1$ let $\mathrm{Li}_2(z)=\sum_{r\ge1}z^r/r^2$, and for $0\le\vartheta\le2\pi$ let
\[
\gamma(\vartheta)=\frac{\vartheta(2\pi-\vartheta)}{4},\qquad \mathrm{Cl}_2(\vartheta)=\sum_{r\ge1}\frac{\sin r\vartheta}{r^2}.
\]
The Fourier series $\sum_{r\ge1}\cos(r\vartheta)/r^2=\frac{\pi^2}6-\frac{\pi\vartheta}2+\frac{\vartheta^2}4$ gives
\begin{equation}\label{eq:Li2circle}
\mathrm{Li}_2(e^{i\vartheta})=\frac{\pi^2}{6}-\gamma(\vartheta)+i\,\mathrm{Cl}_2(\vartheta).
\end{equation}
The Clausen function satisfies $\mathrm{Cl}_2(2\pi-\vartheta)=-\mathrm{Cl}_2(\vartheta)$ and
\begin{equation}\label{eq:clausen}
\mathrm{Cl}_2(\vartheta)=-\int_0^\vartheta\log\Big(2\sin\frac x2\Big)\dd x>0\qquad(0<\vartheta<\pi),
\end{equation}
and $\mathrm{Cl}_2(\pi/2)=G$ is Catalan's constant; see~\cite{Lewin,Zagier}. We also use Euler's identity $\sum_{n\ge0}z^n/(q;q)_n=1/(z;q)_\infty$ for $|z|,|q|<1$~\cite[Chapter~2]{Andrews}, and the transformation law of the Dedekind eta function~\cite[Chapter~3]{Apostol} in the form
\begin{equation}\label{eq:eta}
(e^{-\tau};e^{-\tau})_\infty=\Big(\frac{2\pi}{\tau}\Big)^{1/2}\exp\Big(-\frac{\pi^2}{6\tau}+\frac{\tau}{24}\Big)\big(e^{-4\pi^2/\tau};e^{-4\pi^2/\tau}\big)_\infty\qquad(\mathrm{Re}\,\tau>0),
\end{equation}
with the principal branch of the square root. Finally, we use the Erd\H{o}s--Littlewood--Offord inequality~\cite{Erdos}: if $a_1,\dots,a_k$ are real numbers with $|a_i|\ge1$, then at most $\binom{k}{\lfloor k/2\rfloor}$ of the $2^k$ sums $\sum_i\eps_ia_i$, $\eps_i\in\{\pm1\}$, lie in any open interval of length $2$. Note that $\binom{k}{\lfloor k/2\rfloor}2^{-k}\le k^{-1/2}$.

For $\beta\in\R$ and $t>0$ put
\[
\tau=t(1+i\beta),\qquad q_t=q_t(\beta)=e^{-\tau};
\]
for $\beta=0$ this is $q=e^{-t}$. For $0\le\vartheta\le2\pi$ put
\begin{equation}\label{eq:kappa}
\kappa_\beta(\vartheta)=\frac{\beta\,\mathrm{Cl}_2(\vartheta)-\gamma(\vartheta)}{1+\beta^2},\qquad \psi_\beta(\vartheta)=\frac{\beta\,\gamma(\vartheta)+\mathrm{Cl}_2(\vartheta)}{1+\beta^2}.
\end{equation}
Note that $\kappa_0(\vartheta)=-\gamma(\vartheta)<0$ for $0<\vartheta<2\pi$. The behavior of $q$-products off the real axis is described by the following lemma, proved in Section~\ref{sec:proofs}.

\begin{lemma}\label{lem:poch}
Fix $\beta\in\R$, write $q=q_t(\beta)$, and for $|a|\le1$ put $\log(aq;q)_\infty=\sum_{k\ge1}\mathrm{Log}(1-aq^k)$ with the principal branch.
\begin{enumerate}
\item[(a)] Uniformly for $|a|\le1$ and $0<t\le1$, $\log(aq;q)_\infty=-\mathrm{Li}_2(a)/\tau+O_\beta(\log(2/t))$.
\item[(b)] For fixed $a$ with $|a|=1$, $a\ne1$, as $t\to0^+$,
\[
\frac{(q;q)_\infty}{(aq;q)_\infty}=\Big(\frac{2\pi}{\tau}\Big)^{1/2}(1-a)^{1/2}\exp\Big(-\frac{\pi^2/6-\mathrm{Li}_2(a)}{\tau}\Big)\,(1+o(1)),
\]
with principal square roots. In particular, for $a=e^{i\vartheta}$, $0<\vartheta<2\pi$,
\[
\Big|\frac{(q;q)_\infty}{(e^{i\vartheta}q;q)_\infty}\Big|=\frac{(2\pi|1-e^{i\vartheta}|)^{1/2}}{(1+\beta^2)^{1/4}\,t^{1/2}}\,\exp\Big(\frac{\kappa_\beta(\vartheta)}{t}\Big)(1+o(1)).
\]
\end{enumerate}
\end{lemma}

Asymptotic expansions of this type, uniform in cones, also follow from the Euler--Maclaurin method of~\cite{BJM}; we give a self-contained proof of what is needed.

\section{Main results}

\subsection{Radial limits}

\begin{definition}
Let $L\in\C$. A bounded function $f:\N\to\C$ satisfies condition $(\mathrm{SI}_L)$ if for every fixed $c>0$,
\[
\sum_{x<n\le x+cx/\log x}f(n)=\frac{cLx}{\log x}+o\Big(\frac{x}{\log x}\Big)\qquad(x\to\infty).
\]
A set $S\subseteq\N$ satisfies $(\mathrm{SI}_L)$ if $\mathbf 1_S$ does.
\end{definition}

\begin{theorem}\label{thm:main}
Let $f:\N\to\C$ be bounded and $L\in\C$. The following are equivalent:
\begin{enumerate}
\item[(i)] $\lim_{q\to1^-}\cA_f(q)=L$, that is,
\[
-\lim_{q\to1^-}\sum_{\lambda\in\cD\setminus\{\emptyset\}}\muP(\lambda)f(\sm(\lambda))q^{|\lambda|}=\lim_{q\to1^-}(q;q)_\infty\sum_{\lambda\in\cP}f(\lg(\lambda))q^{|\lambda|}=L;
\]
\item[(ii)] $f$ satisfies $(\mathrm{SI}_L)$.
\end{enumerate}
In particular, for $S\subseteq\N$, $\lim_{q\to1^-}F_S(q)=L$ if and only if
$A_S(x+cx/\log x)-A_S(x)=cLx/\log x+o(x/\log x)$ for every $c>0$.
\end{theorem}

\begin{theorem}\label{thm:abelian}
Let $f:\N\to\C$ be bounded, $d\in\C$, $E(x)=\sum_{n\le x}f(n)-dx$, and suppose
\[
\eps:=\limsup_{x\to\infty}\frac{|E(x)|\log x}{x}<\infty .
\]
Then $\displaystyle\limsup_{q\to1^-}|\cA_f(q)-d|\le\frac{2\eps}{e}$.
\end{theorem}

\begin{corollary}\label{cor:consequences}
Let $f:\N\to\C$ be bounded.
\begin{enumerate}
\item[(a)] If $\sum_{n\le x}f(n)=dx+o(x/\log x)$, then $\lim_{q\to1^-}\cA_f(q)=d$.
\item[(b)] If $\lim_{q\to1^-}\cA_f(q)=L$ exists, then $f$ has mean value $M(f)=L$.
\end{enumerate}
\end{corollary}

\begin{theorem}\label{thm:counter}
\begin{enumerate}
\item[(a)] Let $\theta>0$ and $0<\delta\le\frac12$. There is a set $S\subseteq\N$ with
\begin{equation}\label{eq:AS-osc}
A_S(x)=\frac x2+\frac{\delta}{\theta}\,\frac{x}{\log x}\,\sin\!\Big(\frac{\theta(\log x)^2}{2}\Big)+O_\theta\Big(\frac{x}{(\log x)^2}\Big),
\end{equation}
such that, as $t\to0^+$,
\begin{equation}\label{eq:FS-osc}
F_S(e^{-t})=\frac12+\delta\,\mathrm{Re}\big(e^{i\theta w(t)}\,\Gamma(1-i\theta)\big)+o(1).
\end{equation}
In particular $\limsup_{q\to1^-}F_S(q)$ and $\liminf_{q\to1^-}F_S(q)$ equal $\frac12\pm\delta\,(\pi\theta/\sinh\pi\theta)^{1/2}$.
\item[(b)] Let $h:[2,\infty)\to(0,\infty)$ satisfy $h(x)\log x/x\to\infty$. There is a set $S\subseteq\N$ with $|A_S(x)-x/2|\le h(x)$ for all sufficiently large $x$ such that
\[
\liminf_{q\to1^-}F_S(q)=0\qquad\text{and}\qquad\limsup_{q\to1^-}F_S(q)=1 .
\]
\end{enumerate}
\end{theorem}

Taking $\theta=1$ and $\delta=\min(\eps,1)/2$ in part~(a) produces, for every $\eps>0$, a set with $|A_S(x)-x/2|\le\eps x/\log x$ for all large $x$ for which $F_S(q)$ has no limit. Together with Corollary~\ref{cor:consequences}(a), this shows that the threshold $x/\log x$ in Theorem~\ref{thm:abelian} cannot be improved. Part~(b) shows that once the error term is allowed to exceed $x/\log x$ by a factor tending to infinity, nothing beyond the trivial bounds $0\le F_S\le 1$ survives.

\begin{corollary}\label{cor:ND}
There is a set $S\subseteq\N$ of natural density $\tfrac12$ such that $\liminf_{q\to1^-}F_S(q)=0$ and $\limsup_{q\to1^-}F_S(q)=1$. Consequently, the existence of $M(f)$ does not imply convergence of $\cA_f(q)$ as $q\to 1^-$, even for indicator functions and radial approach, and the formula~\eqref{eq:OSWlimit} does not hold for all sets possessing a natural density.
\end{corollary}

Thus Corollary~\ref{cor:consequences}(a) is not a consequence of Frobenius' theorem: by Corollary~\ref{cor:ND}, its conclusion fails under the hypothesis $M(f)=d$ alone, and the additional input is the size $o(x/\log x)$ of the error term.

\begin{remark}\label{rem:onec}
In condition $(\mathrm{SI}_L)$ the quantifier ``for every $c>0$'' cannot be replaced by a single value of $c$. Given $c_0>0$, apply Theorem~\ref{thm:counter}(a) with $\theta=2\pi/c_0$. By~\eqref{eq:window-osc} below, $A_S(x+c_0x/\log x)-A_S(x)=c_0x/(2\log x)+o(x/\log x)$, yet $F_S(q)$ oscillates with amplitude $\delta(2\pi^2/(c_0\sinh(2\pi^2/c_0)))^{1/2}>0$. The obstruction is that the Fourier transform of $\mathbf 1_{[0,c_0]}$ vanishes at $\xi=2\pi/c_0=\theta$, so a single window length cannot detect oscillation of $\Phi_{\mathbf 1_S}$ at that frequency.
\end{remark}

\begin{example}\label{ex:examples}
(1) If $S$ is a union of $s$ residue classes modulo $m$, then $A_S(x)=sx/m+O(1)$. For $k\ge2$ let $Q_k$ be the set of $k$-free integers; then $A_{Q_k}(x)=\sum_{d\le x^{1/k}}\mu(d)\lfloor x/d^k\rfloor=x/\zeta(k)+O(x^{1/k})$. Corollary~\ref{cor:consequences}(a) gives $F_S(q)\to s/m$ and $F_{Q_k}(q)\to1/\zeta(k)$, recovering the corresponding radial formulas of~\cite{OSW1,OSW2}.

(2) The prime number theorem with the classical error term gives, for some $c_1>0$, the bound $\sum_{n\le x}\mu(n)\ll x\exp(-c_1\sqrt{\log x})$ (see~\cite[Chapter~6]{MV}). Hence $\cA_\mu(q)\to0$, that is,
\[
\lim_{q\to1^-}\sum_{\lambda\in\cD\setminus\{\emptyset\}}\muP(\lambda)\,\mu(\sm(\lambda))\,q^{|\lambda|}=0 .
\]

(3) Let $\mathbb P$ be the set of primes. Here $A_{\mathbb P}(x)\sim x/\log x$, so Theorem~\ref{thm:abelian} gives only $\limsup F_{\mathbb P}(q)\le 2/e$. However, $\pi(x)=x/\log x+o(x/\log x)$ gives $\pi(x+cx/\log x)-\pi(x)=o(x/\log x)$, so $(\mathrm{SI}_0)$ holds and Theorem~\ref{thm:main} gives $F_{\mathbb P}(q)\to0$. Thus condition $(\mathrm{SI}_L)$ is strictly weaker than the hypothesis of Corollary~\ref{cor:consequences}(a).
\end{example}

\subsection{Sparse sets and primes}
Alladi's formula~\eqref{eq:alladi} concerns primes. Every set $S$ of primes satisfies $(\mathrm{SI}_0)$, since $0\le A_S(x+H)-A_S(x)\le\pi(x+H)-\pi(x)=o(x/\log x)$ for $H=cx/\log x$; hence $F_S(q)\to0$ by Theorem~\ref{thm:main}. For sets of local density $(\log x)^{-\alpha}$ the correct analogue carries the normalization $(\log(1/(1-q)))^{\alpha}$, and the short-interval characterization persists.

\begin{theorem}\label{thm:primes}
Let $S\subseteq\N$, $\alpha\ge0$ and $\delta\ge0$. The following are equivalent:
\begin{enumerate}
\item[(i)] $\displaystyle\lim_{q\to1^-}\Big(\log\frac{1}{1-q}\Big)^{\alpha}F_S(q)=\delta$;
\item[(ii)] for every fixed $c>0$,
\[
A_S\Big(x+\frac{cx}{\log x}\Big)-A_S(x)=\frac{c\,\delta\,x}{(\log x)^{1+\alpha}}+o\Big(\frac{x}{(\log x)^{1+\alpha}}\Big)\qquad(x\to\infty).
\]
\end{enumerate}
\end{theorem}

For $\alpha=0$ this is the case $f=\mathbf 1_S$ of Theorem~\ref{thm:main}. For $\alpha>0$ the relevant weight $\mathbf 1_S(n)(\log n)^\alpha$ is unbounded, and the Tauberian condition that boundedness provides in Theorem~\ref{thm:main} is here derived from positivity rather than assumed. Since the kernel is bounded below on its central window, (i) already forces $S$ to have $O(x/(\log x)^{1+\alpha})$ elements in every interval $(x,x+x/\log x]$ (Section~\ref{sec:primesproof}).

\begin{corollary}\label{cor:primesuff}
Let $S\subseteq\N$, $\alpha\ge0$, $\delta\ge0$, and $\mathrm{li}_\alpha(x)=\int_2^x(\log y)^{-\alpha}\dd y$, so that $\mathrm{li}_1=\mathrm{li}$. If $A_S(x)=\delta\,\mathrm{li}_\alpha(x)+\delta'x/(\log x)^{1+\alpha}+o(x/(\log x)^{1+\alpha})$ for some $\delta'\in\R$, then $(\log(1/(1-q)))^{\alpha}F_S(q)\to\delta$ as $q\to1^-$.
\end{corollary}

For $\alpha=1$ and $S=\mathbb P$, Theorem~\ref{thm:primes} states that the \emph{partition-theoretic prime number theorem} $\log(1/(1-q))F_{\mathbb P}(q)\to1$ is equivalent to $\pi(x+cx/\log x)-\pi(x)\sim cx/(\log x)^2$ for every $c>0$. Since $\pi(x)=\mathrm{li}(x)+O(x\exp(-c_1\sqrt{\log x}))$ for some $c_1>0$, Corollary~\ref{cor:primesuff} shows that both statements hold. The convergence is slow, with relative error of order $\log L_t/L_t$, because the Gumbel window sits at $x_t$, where $\log x_t=L_t+\log L_t$. For $\alpha=\frac12$ the corollary applies to the integers that are sums of two squares, whose counting function has the form $bx(\log x)^{-1/2}(1+c'(\log x)^{-1}+O((\log x)^{-2}))$ by Landau's method, as established in~\cite[(5), (8$'$)]{Shanks}; the limit is the Landau--Ramanujan constant $b$.

\begin{corollary}\label{cor:chebotarev}
Let $K/\mathbb Q$ be a finite Galois extension, let $C\subseteq\mathrm{Gal}(K/\mathbb Q)$ be a conjugacy class, and let $\mathbb P_C$ be the set of primes $p$ unramified in $K$ whose Frobenius class is $C$. Then
\[
\lim_{q\to1^-}\log\frac1{1-q}\cdot\Big(-\sum_{\substack{\lambda\in\cD\setminus\{\emptyset\}\\ \sm(\lambda)\in\mathbb P_C}}\muP(\lambda)q^{|\lambda|}\Big)=\frac{|C|}{[K:\mathbb Q]}.
\]
In particular, for coprime integers $\ell$ and $k\ge1$,
\[
\lim_{q\to1^-}\log\frac1{1-q}\cdot\Big(-\sum_{\substack{\lambda\in\cD\setminus\{\emptyset\},\ \sm(\lambda)\ \mathrm{prime}\\ \sm(\lambda)\equiv\ell\ (\mathrm{mod}\ k)}}\muP(\lambda)q^{|\lambda|}\Big)=\frac1{\varphi(k)} .
\]
\end{corollary}

This is the partition-theoretic counterpart of~\eqref{eq:alladi} and of its Chebotarev generalizations~\cite{Dawsey,SweetingWoo,KMS}: the least prime factor is replaced by the smallest part, $\mu$ by $\muP$, and the harmonic weight $1/n$ by the normalization $\log(1/(1-q))$.

\subsection{Approach off the real axis}
Let $\tau$ and $q_t=q_t(\beta)$ be as in Section~\ref{sec:dilog}. As $t\to0^+$, $|1-q_t|/(1-|q_t|)\to(1+\beta^2)^{1/2}$. Hence the path $\{q_t(\beta)\}$, a ray in the $\tau$-plane and a logarithmic spiral in the $q$-plane, lies eventually in the Stolz sector $|1-q|\le M_0(1-|q|)$ whenever $(1+\beta^2)^{1/2}<M_0$, and every Stolz sector with $M_0>1$ contains such paths with $\beta\neq0$. The identities of Lemma~\ref{lem:identities} hold for complex $|q|<1$, but the weights $\omega_n(q_t)$ are no longer positive. We use $\gamma$, $\mathrm{Cl}_2$ and $\kappa_\beta$ from~\eqref{eq:kappa}.

\begin{theorem}\label{thm:periodic}
Let $f$ be periodic modulo $m$, with $\widehat f(j)=\frac1m\sum_{r=1}^mf(r)e^{-2\pi ijr/m}$, and let
\[
\kappa^*=\max\big\{\kappa_\beta(2\pi j/m):\ 1\le j\le m-1,\ \widehat f(j)\ne0\big\}\qquad(\max\emptyset=-\infty).
\]
\begin{enumerate}
\item[(a)] If $\kappa^*<0$, then $\cA_f(q_t)\to\widehat f(0)=M(f)$ as $t\to0^+$.
\item[(b)] If $\kappa^*\ge0$, then $\limsup_{t\to0^+}t^{1/2}e^{-\kappa^*/t}\,|\cA_f(q_t)|>0$; in particular $\cA_f(q_t)$ is unbounded.
\end{enumerate}
\end{theorem}

Since $\kappa_0<0$ on $(0,2\pi)$, part~(a) always applies for $\beta=0$, in accordance with Corollary~\ref{cor:consequences}(a). For residue classes the criterion becomes explicit.

\begin{corollary}\label{cor:AP}
For $m\ge3$ let
\[
\beta_m=\min_{1\le j<m/2}\frac{\gamma(2\pi j/m)}{\mathrm{Cl}_2(2\pi j/m)} .
\]
Let $r\in\Z$ and $S_{r,m}=\{n\ge1:\ n\equiv r\ (\mathrm{mod}\ m)\}$. If $|\beta|<\beta_m$, then $F_{S_{r,m}}(q_t)\to1/m$; if $|\beta|\ge\beta_m$, then $F_{S_{r,m}}(q_t)$ is unbounded as $t\to0^+$. For $m\le2$, $F_{S_{r,m}}(q_t)\to1/m$ for every $\beta$. Moreover
\[
\beta_m\le\frac{\pi}{2\,(1+\log(m/2\pi))}\longrightarrow0\qquad(m\to\infty),
\]
and
\[
\begin{aligned}
\beta_3&=\frac{2\pi^2}{9\,\mathrm{Cl}_2(2\pi/3)}=3.24143587\ldots,\qquad \beta_4=\frac{3\pi^2}{16\,G}=2.02032787\ldots,\\
\beta_6&=\frac{5\pi^2}{36\,\mathrm{Cl}_2(\pi/3)}=1.35059828\ldots,
\end{aligned}
\]
where $G$ is Catalan's constant.
\end{corollary}

Thus for every $\beta\ne0$ the density formula fails along $q_t(\beta)$ for all residue classes modulo every sufficiently large $m$. The ratio in~\cite[Lemma~2.5]{OSW1} is the case of a single frequency: by Lemma~\ref{lem:poch}(b), $|(q_t;q_t)_\infty/(e^{i\vartheta}q_t;q_t)_\infty|$ tends to $0$ if $\kappa_\beta(\vartheta)<0$ and to $\infty$ if $\kappa_\beta(\vartheta)\ge0$.

Arithmetic structure is not the source of the instability. Off the real axis the largest weights become exponentially large, and random sets of bounded discrepancy fail.

\begin{theorem}\label{thm:instability}
For $\beta\in\R$ put
\[
\Lambda_\beta(s)=-\mathrm{Re}\,\frac{\mathrm{Li}_2\big(e^{-s(1+i\beta)}\big)}{1+i\beta}\quad(s\ge0),\qquad \Psi_\beta=\sup_{s>0}\Lambda_\beta(s),
\]
Let $\beta\ne0$.
\begin{enumerate}
\item[(a)] $\Psi_\beta\ge(e^{-s_\beta}-e^{-2s_\beta})(1+\beta^2)^{-1/2}>0$ with $s_\beta=(\pi-\arctan|\beta|)/|\beta|$, and
\[
\max_{n\ge1}|\omega_n(q_t)|=\exp\Big(\frac{\Psi_\beta}{t}+O_\beta\big(\log(2/t)\big)\Big)\qquad(0<t\le1).
\]
\item[(b)] For each $k\ge1$ choose, independently and with probability $\frac12$ each, one element of $\{2k-1,2k\}$, and let $S$ be the resulting random set. Then every outcome satisfies $|A_S(x)-x/2|\le1$ for all $x$, and hence $F_S(q)\to\frac12$ as $q\to1^-$. Almost surely, $\limsup_{t\to0^+}|F_S(q_t(\beta))|=\infty$ simultaneously for all $\beta$ in any prescribed countable subset of $\R\setminus\{0\}$.
\end{enumerate}
\end{theorem}

For $\beta=0$ one has $\Lambda_0(s)=-\mathrm{Li}_2(e^{-s})<0$, and the weights are the positive Gumbel weights of Lemma~\ref{lem:gumbel}. On the real axis every set of bounded discrepancy satisfies the formula, by Corollary~\ref{cor:consequences}(a). Off the real axis bounded discrepancy decides nothing: the even numbers satisfy $F_S(q_t)\to\frac12$ for every $\beta$ by Corollary~\ref{cor:AP}, whereas the random sets of Theorem~\ref{thm:instability}(b), with the same discrepancy, do not. Convergence then depends on the fine structure of $S$ and not only on its counting function.

\begin{corollary}\label{cor:stolz}
In every Stolz sector at $1$ with $M_0>1$ there are paths along which $F_S(q)$ is unbounded for some arithmetic progression $S$, and for almost every random set $S$ as in Theorem~\ref{thm:instability}(b). In particular, the Stolz-sector statements~\cite[Theorem~1.1 and Corollaries~1.2 and~1.3]{Schneider} fail already for periodic functions and for sets of bounded discrepancy, and the limits in~\cite[Lemma~2.5 and Theorem~1.3]{OSW1}, for $q\to1$, do not hold for unrestricted approach from within the unit disk.
\end{corollary}

\section{Proofs}\label{sec:proofs}

\subsection{The Gumbel approximation}

\begin{lemma}\label{lem:gumbel}
For $0<t\le1$ and $q=e^{-t}$,
\[
\sum_{n\ge1}\Big|\omega_n(q)-t\,g\big(nt-L_t\big)\Big|\le 8t .
\]
\end{lemma}

\begin{proof}
Fix $n\ge1$ and put $y=e^{-nt}\in(0,1)$. Since $te^{L_t}=1$,
\[
t\,g(nt-L_t)=t\,e^{-nt+L_t}\exp\!\big(-e^{-nt+L_t}\big)=y\,e^{-h},\qquad h:=\frac yt,
\]
while $\omega_n(q)=y\,e^{-H}$ with
\[
H:=-\log(q^{n+1};q)_\infty=\sum_{m>n}\sum_{k\ge1}\frac{q^{km}}{k}=\sum_{k\ge1}\frac1k\,\frac{e^{-knt}}{e^{kt}-1}.
\]
The term $k=1$ equals $y/(e^t-1)$, and $0<\frac1t-\frac1{e^t-1}<\frac12$ for $t>0$. For $k\ge2$ we use $e^{kt}-1\ge kt$. Hence
\[
-\frac y2\le H-h\le\frac1t\sum_{k\ge2}\frac{y^k}{k^2}\le\frac{y^2}{t}.
\]
In particular $\min(H,h)\ge h-\frac12$. By the mean value theorem, $|e^{-H}-e^{-h}|\le e^{-\min(H,h)}|H-h|$, and therefore
\[
\big|\omega_n(q)-t\,g(nt-L_t)\big|\le e^{1/2}\,y\,e^{-y/t}\Big(\frac y2+\frac{y^2}t\Big)=e^{1/2}\,t^2\,p(s_n),
\]
where $s_n=e^{-nt}/t$ and $p(s)=s^2(\frac12+s)e^{-s}$. The logarithmic derivative $2/s+1/(\frac12+s)-1$ of $p$ is strictly decreasing, so $p$ is unimodal on $(0,\infty)$, and $\max p\le2$. As $x\mapsto s(x)=e^{-xt}/t$ is decreasing, $Q(x):=t^2p(s(x))$ is unimodal on $[0,\infty)$. For a nonnegative unimodal function, $\sum_{n\ge1}Q(n)\le\int_0^\infty Q+\max Q$. The substitution $s=s(x)$, $\dd x=-\dd s/(ts)$, gives
\[
\int_0^\infty Q(x)\dd x=t\int_0^{1/t}\frac{p(s)}{s}\dd s\le t\int_0^\infty s\Big(\frac12+s\Big)e^{-s}\dd s=\frac{5t}{2}.
\]
Hence the sum in the lemma is at most $e^{1/2}(\frac52t+2t^2)\le\frac92e^{1/2}t<8t$. The constant $8$ is not optimal; the arguments below use only the order $O(t)$.
\end{proof}

\subsection{Reduction to a convolution}

\begin{lemma}\label{lem:reduction}
There is a function $\eta:(0,1/e)\to[0,\infty)$ with $\eta(t)\to0$ as $t\to0^+$ such that, for every bounded $f$ and every $t\in(0,1/e)$,
\[
\big|\cA_f(e^{-t})-J_f(w(t))\big|\le\eta(t)\,\norm f .
\]
\end{lemma}

\begin{proof}
Write $k_t(x)=t\,g(tx-L_t)$. We compare $\cA_f(e^{-t})=\sum f(n)\omega_n$ with $J_f(w(t))$ in four steps.

\emph{Step 1.} By Lemma~\ref{lem:gumbel}, $|\cA_f(e^{-t})-\sum_{n\ge1}f(n)k_t(n)|\le8t\norm f$.

\emph{Step 2.} Since $\psi_f=f(n)$ on $(n-1,n]$,
\[
\Big|\sum_{n\ge1}f(n)k_t(n)-\int_0^\infty\psi_f(x)k_t(x)\dd x\Big|\le\norm f\sum_{n\ge1}\int_{n-1}^n|k_t(n)-k_t(x)|\dd x\le\norm f\int_0^\infty|k_t'(y)|\dd y,
\]
because $|k_t(n)-k_t(x)|\le\int_{n-1}^n|k_t'|$ for $x\in(n-1,n]$. By~\eqref{eq:gprime}, $\int_0^\infty|k_t'|\le t\int_\R|g'|=2t/e$.

\emph{Step 3.} $\int_0^1|\psi_f|k_t\le t\,g(0)\norm f=t\norm f/e$.

\emph{Step 4.} By~\eqref{eq:Jx},
\[
\Big|\int_1^\infty\psi_f(x)k_t(x)\dd x-J_f(w(t))\Big|\le\norm f\,D(t),
\]
where
\[
D(t):=\int_1^\infty\Big|t\,g(tx-L_t)-v'(x)\,g\big(v(x)-w(t)\big)\Big|\dd x .
\]
Substitute $x=z_u:=(u+L_t)/t$, so that $x\ge1$ corresponds to $u\ge t-L_t$, $z_0=x_t$ and $\log z_u=L_t+\log(u+L_t)$. Put
\[
\sigma_t(u):=v(z_u)-w(t),\qquad \sigma_t'(u)=\frac{v'(z_u)}{t}=\frac{L_t+\log(u+L_t)}{u+L_t}\qquad(u\ge t-L_t).
\]
Then $t\,g(tx-L_t)\dd x=g(u)\dd u$ and $v'(x)g(v(x)-w(t))\dd x=h_t(u)\dd u$ with $h_t(u):=g(\sigma_t(u))\sigma_t'(u)$ for $u\ge t-L_t$ and $h_t(u):=0$ otherwise. Thus
\[
D(t)\le\int_\R|h_t(u)-g(u)|\dd u .
\]
Here $h_t\ge0$ because $\log z_u\ge0$. As $\sigma_t$ increases from $\sigma_t(t-L_t)=v(1)-w(t)=-w(t)$ to $\infty$,
\[
\int_\R h_t=\int_{-w(t)}^\infty g=1-\mathcal G(-w(t))\longrightarrow1=\int_\R g\qquad(t\to0^+).
\]
For fixed $u\in\R$ and $L_t\to\infty$, $\sigma_t'(u')=(L_t+\log(u'+L_t))/(u'+L_t)\to1$ uniformly for $u'$ between $0$ and $u$. Since $\sigma_t(0)=0$, it follows that $\sigma_t(u)=\int_0^u\sigma_t'\to u$, and hence $h_t(u)\to g(u)$ by continuity of $g$. By Lemma~\ref{lem:scheffe}, applied along an arbitrary sequence $t_j\to0^+$, $D(t)\to0$.

Collecting the four steps, the lemma holds with $\eta(t)=8t+2t/e+t/e+\int_\R|h_t-g|$.
\end{proof}

\subsection{Windows in the new coordinate}

\begin{lemma}\label{lem:window}
Let $f$ be bounded and $c>0$ fixed. As $X\to\infty$,
\[
\frac{\log X}{X}\sum_{X<n\le X+cX/\log X}f(n)=\int_{v(X)}^{v(X)+c}\Phi_f(v)\dd v+o(1).
\]
\end{lemma}

\begin{proof}
Let $H=cX/\log X$ and $X\ge e$. First, $\sum_{X<n\le X+H}f(n)$ and $\int_X^{X+H}\psi_f$ differ by at most $2\norm f$, since only the two unit cells $(n-1,n]$ containing $X$ or $X+H$ are cut. Second, with $x=\chi(v)$,
\[
\int_X^{X+H}\psi_f(x)\dd x=\int_{v(X)}^{v(X+H)}\Phi_f(v)\chi'(v)\dd v .
\]
On this range, $\chi'(v)=x/\log x$ lies between $X/\log X$ and $(X+H)/\log(X+H)\le(X/\log X)(1+c/\log X)$, because $x/\log x$ is increasing for $x\ge e$. Third, since $\log y/y$ is decreasing for $y\ge e$,
\[
c\,\frac{X\log(X+H)}{(X+H)\log X}\le v(X+H)-v(X)=\int_X^{X+H}\frac{\log y}{y}\dd y\le c ,
\]
so $v(X+H)-v(X)=c+o(1)$. Combining the three observations,
\[
\frac{\log X}{X}\sum_{X<n\le X+H}f(n)=\int_{v(X)}^{v(X+H)}\Phi_f+O\Big(\frac{c^2\norm f}{\log X}+\frac{\norm f\log X}{X}\Big)=\int_{v(X)}^{v(X)+c}\Phi_f+o(1),
\]
where the last step uses $|\int_{v(X)+c}^{v(X+H)}\Phi_f|\le\norm f\,|v(X+H)-v(X)-c|$.
\end{proof}

Since $X\mapsto v(X)$ is an increasing bijection of $[1,\infty)$ onto $[0,\infty)$, Lemma~\ref{lem:window} shows that $f$ satisfies $(\mathrm{SI}_L)$ if and only if
\begin{equation}\label{eq:SIv}
\lim_{w\to\infty}\int_w^{w+c}\Phi_f(v)\dd v=cL\qquad\text{for every } c>0 .
\end{equation}

\subsection{Proof of Theorem~\ref{thm:main}}
By Lemma~\ref{lem:identities} the three limits in~(i) are limits of the same function $\cA_f$. By Lemma~\ref{lem:reduction}, and since $t\mapsto w(t)$ is a decreasing homeomorphism of $(0,1/e)$ onto $(\frac12,\infty)$, condition~(i) is equivalent to
\begin{equation}\label{eq:Jlimit}
\lim_{w\to\infty}J_f(w)=L .
\end{equation}
It therefore suffices to prove that~\eqref{eq:Jlimit} is equivalent to~\eqref{eq:SIv}.

Assume~\eqref{eq:SIv}. Lemma~\ref{lem:abelstep}, applied with $\Theta=\Phi_f-L$ and $k=g$, gives $J_f(w)-L\int_0^\infty g(v-w)\dd v\to0$. Since $\int_0^\infty g(v-w)\dd v=1-\mathcal G(-w)\to1$, this proves~\eqref{eq:Jlimit}.

Conversely, assume~\eqref{eq:Jlimit}. Define $\widetilde\Phi\in L^\infty(\R)$ by $\widetilde\Phi=\Phi_f$ on $[0,\infty)$ and $\widetilde\Phi=L$ on $(-\infty,0)$, and let $K(z)=g(-z)$. By~\eqref{eq:fourier}, $\widehat K(\xi)=\Gamma(1-i\xi)\neq0$ for all real $\xi$, and $\widehat K(0)=1$. Moreover
\[
(K*\widetilde\Phi)(w)=\int_\R g(v-w)\widetilde\Phi(v)\dd v=J_f(w)+L\,\mathcal G(-w).
\]
As $w\to+\infty$ this tends to $L$ by~\eqref{eq:Jlimit}. As $w\to-\infty$ we have $|J_f(w)|\le\norm f(1-\mathcal G(-w))\to0$ and $L\mathcal G(-w)\to L$. Hence $(K*\widetilde\Phi)(w)\to L\widehat K(0)$ as $|w|\to\infty$. Theorem~\ref{thm:wiener} with $k=\mathbf 1_{[-c,0]}$ yields
\[
(k*\widetilde\Phi)(w)=\int_w^{w+c}\widetilde\Phi(v)\dd v\longrightarrow cL\qquad(w\to+\infty),
\]
which is~\eqref{eq:SIv}. The statement for sets is the case $f=\mathbf 1_S$. \qed

\subsection{Proof of Theorem~\ref{thm:abelian}}
Let $\eps'>\eps$ and choose $N_0\ge3$ with $|E(n)|\le\eps'n/\log n$ for $n\ge N_0$. Let $t\in(0,1/e)$, $L=L_t$ and $k_t(x)=t\,g(tx-L)$. By Lemma~\ref{lem:gumbel} and~\eqref{eq:mass},
\[
\cA_f(e^{-t})-d=\sum_{n\ge1}(f(n)-d)k_t(n)+O\big((\norm f+|d|)t\big)-d\,(e^{-t};e^{-t})_\infty ,
\]
and $(e^{-t};e^{-t})_\infty\to0$. Put $E(0)=0$, so that $f(n)-d=E(n)-E(n-1)$. Since $E(n)=O(n)$ and $k_t(n)$ decays exponentially in $n$, partial summation gives
\[
\sum_{n\ge1}(f(n)-d)k_t(n)=\sum_{n\ge1}E(n)\big(k_t(n)-k_t(n+1)\big),\qquad |k_t(n)-k_t(n+1)|\le\int_n^{n+1}t^2|g'(ty-L)|\dd y .
\]
The terms with $n<N_0$ contribute at most $N_0^2(\norm f+|d|)t^2\max|g'|\to0$. For $n\ge N_0$ and $y\in[n,n+1]$ we have $|E(n)|\le\eps'n/\log n\le\eps'y/\log y$, so these terms contribute at most $\eps'T(t)$, where
\[
T(t)=\int_{N_0}^\infty\frac{y}{\log y}\,t^2|g'(ty-L)|\dd y=\int_{tN_0-L}^\infty r_t(u)\,|g'(u)|\dd u,\qquad r_t(u)=\frac{u+L}{L+\log(u+L)} ,
\]
after the substitution $u=ty-L$. For fixed $u$, $r_t(u)\to1$ as $t\to0^+$. If $L\ge2$ and $u\ge-L/2$, then $\log(u+L)\ge0$ and $r_t(u)\le1+|u|$, and $(1+|u|)|g'(u)|$ is integrable. If $tN_0-L\le u<-L/2$, then $y\ge N_0\ge3$ gives $r_t(u)=ty/\log y\le u+L\le L$. In addition, $|g'(u)|=g(u)|e^{-u}-1|\le e^{2|u|}\exp(-e^{|u|})$ for $u<0$. Hence this range contributes at most
\[
L\int_{L/2}^\infty e^{2s}\exp(-e^s)\dd s=L\int_{e^{L/2}}^\infty r e^{-r}\dd r\to0 .
\]
By dominated convergence and~\eqref{eq:gprime}, $T(t)\to\int_\R|g'|=2/e$. Therefore $\limsup_{t\to0^+}|\cA_f(e^{-t})-d|\le2\eps'/e$ for every $\eps'>\eps$. \qed

\subsection{Proof of Corollary~\ref{cor:consequences}}
Part~(a) is the case $\eps=0$ of Theorem~\ref{thm:abelian}. For~(b), Theorem~\ref{thm:main} gives $(\mathrm{SI}_L)$, and we use only $c=1$. Let $\eps>0$ and choose $Y\ge e$ such that for $X\ge Y$,
\[
\Big|\sum_{X<n\le X+X/\log X}f(n)-\frac{LX}{\log X}\Big|\le\frac{\eps X}{\log X}.
\]
Define $X_0=Y$ and $X_{j+1}=X_j+X_j/\log X_j$. Given $x>Y$, let $J$ be maximal with $X_J\le x$. The intervals $(X_j,X_{j+1}]$, $0\le j<J$, partition $(Y,X_J]$. Summing over them gives $|\sum_{Y<n\le X_J}f(n)-L(X_J-Y)|\le\eps(X_J-Y)$. Since $x-X_J<X_J/\log X_J=o(x)$,
\[
\Big|\sum_{n\le x}f(n)-Lx\Big|\le\eps x+o(x),
\]
and letting $\eps\to0$ gives $M(f)=L$. \qed

\subsection{Proof of Theorem~\ref{thm:counter}(a)}
Let $\rho(n)=\frac12+\delta\cos(\theta v(n))\in[0,1]$, $R(0)=0$, $R(n)=\sum_{m\le n}\rho(m)$, and
\[
S=\{n\ge1:\ \lfloor R(n)\rfloor>\lfloor R(n-1)\rfloor\}.
\]
Since $0\le\rho\le1$, the difference $\lfloor R(n)\rfloor-\lfloor R(n-1)\rfloor$ is $0$ or $1$, so $A_S(N)=\lfloor R(N)\rfloor$ for every integer $N\ge0$. Hence $A_S(x)=\lfloor x\rfloor/2+\delta\sum_{n\le x}\cos(\theta v(n))+O(1)$.

The function $\phi(y)=\cos(\theta v(y))$ has total variation at most $\theta v(x)=\theta(\log x)^2/2$ on $[1,x]$, so $\sum_{n\le x}\phi(n)=\int_e^x\phi(y)\dd y+O_\theta((\log x)^2)$. Write $\ell=\log y$. Using $\phi(y)=\frac{y}{\theta\ell}\,\frac{\dd}{\dd y}\sin(\theta v(y))$ and integrating by parts,
\[
\int_e^x\phi=\Big[\frac{y\sin(\theta v(y))}{\theta\ell}\Big]_e^x-\frac1\theta\int_e^x\sin(\theta v(y))\,\frac{\ell-1}{\ell^2}\dd y .
\]
In the last integral write $\sin(\theta v(y))=-\frac{y}{\theta\ell}\frac{\dd}{\dd y}\cos(\theta v(y))$ and integrate by parts again. The boundary term is $O_\theta(x/(\log x)^2)$. The remaining integrand is $\cos(\theta v(y))$ times
\[
\frac{\dd}{\dd y}\,\frac{y(\ell-1)}{\ell^3}=\frac1{\ell^2}-\frac3{\ell^3}+\frac3{\ell^4}=O\Big(\frac1{\ell^2}\Big),
\]
and $\int_e^x\ell^{-2}\dd y=O(x/(\log x)^2)$. Hence $\int_e^x\phi=x\sin(\theta v(x))/(\theta\log x)+O_\theta(x/(\log x)^2)$, which proves~\eqref{eq:AS-osc}.

Now fix $c>0$, let $X\to\infty$, and put $H=cX/\log X$ and $w=v(X)$. Since $(X+H)/\log(X+H)=(X/\log X)(1+O(1/\log X))$ and, by the proof of Lemma~\ref{lem:window}, $v(X+H)=w+c+o(1)$,~\eqref{eq:AS-osc} gives
\begin{equation}\label{eq:window-osc}
\begin{aligned}
\frac{\log X}{X}\big(A_S(X+H)-A_S(X)\big)&=\frac c2+\frac\delta\theta\big(\sin(\theta(w+c))-\sin(\theta w)\big)+o(1)\\
&=\int_w^{w+c}r(v)\dd v+o(1),
\end{aligned}
\end{equation}
where $r(v)=\frac12+\delta\cos(\theta v)$. By Lemma~\ref{lem:window}, $\int_w^{w+c}(\Phi_{\mathbf 1_S}-r)\to0$ as $w\to\infty$ for every $c>0$. Lemma~\ref{lem:abelstep} with $\Theta=\Phi_{\mathbf 1_S}-r$ and $k=g$ gives $J_{\mathbf 1_S}(w)-\int_0^\infty r(v)g(v-w)\dd v\to0$. By~\eqref{eq:fourier},
\[
\int_0^\infty r(v)g(v-w)\dd v=\int_{-w}^\infty r(w+u)g(u)\dd u=\frac12+\delta\,\mathrm{Re}\Big(e^{i\theta w}\int_\R e^{i\theta u}g(u)\dd u\Big)+O(\mathcal G(-w)),
\]
and $\int_\R e^{i\theta u}g(u)\dd u=\Gamma(1-i\theta)$. Lemma~\ref{lem:reduction} now gives~\eqref{eq:FS-osc}. Since $w(t)$ runs continuously through all sufficiently large reals as $t\to0^+$, the upper and lower limits of $F_S$ are $\frac12\pm\delta|\Gamma(1-i\theta)|$, and $|\Gamma(1-i\theta)|^2=\pi\theta/\sinh(\pi\theta)$. \qed

\subsection{Proof of Theorem~\ref{thm:counter}(b)}
\emph{Normalization of $h$.} Put $h_1(x)=\inf_{y\ge x}\min(h(y),y)$ for $x\ge3$. Then $h_1\le h$, and $h_1$ is nondecreasing. Write $\nu(y)=h(y)\log y/y$. Since $y/\log y$ is increasing for $y\ge e$, we have $h_1(x)\ge\inf_{y\ge x}\min(\nu(y),\log y)\cdot x/\log x$, so $h_1(x)\log x/x\to\infty$; in particular $h_1(x)>0$ for all large $x$, which is all that is used below. We may therefore assume that $h$ is nondecreasing and $h(x)\le x$, and we write $\nu(x)=h(x)\log x/x\to\infty$.

\emph{Construction.} Choose even integers $X_1<X_2<\cdots$ with $h(X_1)\ge8$ and $X_{k+1}\ge4X_k$, and put $B_k=\lfloor h(X_k)\rfloor$, so that $8\le B_k\le X_k$. Let
\[
I_k^-=(X_k,X_k+B_k]\cap\N,\qquad I_k^+=(X_k+B_k,X_k+2B_k]\cap\N .
\]
These blocks lie in $(X_k,3X_k]$ and are pairwise disjoint. Define
\[
S=\Big(2\N\setminus\bigcup_k(I_k^-\cup I_k^+)\Big)\cup\bigcup_kI_k^+ .
\]

\emph{Counting function.} Any $2B_k$ consecutive integers contain exactly $B_k$ even integers, and $\#I_k^+=B_k$. Hence $A_S(x)=\lfloor x/2\rfloor$ for every $x\ge 1$ outside $\bigcup_k(X_k,X_k+2B_k)$. For $x=X_k+s$ with $0<s\le2B_k$, since $X_k$ is even,
\[
A_S(x)-\Big\lfloor\frac x2\Big\rfloor=\max(0,\lfloor s\rfloor-B_k)-\Big\lfloor\frac s2\Big\rfloor\in\Big[-\frac{B_k}2-1,\ 1\Big].
\]
Thus $|A_S(x)-x/2|\le B_k/2+2\le h(X_k)\le h(x)$ for $x\in(X_k,X_k+2B_k]$, and $|A_S(x)-x/2|\le\frac12\le h(x)$ for all other $x\ge X_1$.

\emph{Oscillation.} Let $a_k=X_k+B_k/2$ and choose $t_k\in(0,1/e)$ with $x_{t_k}=a_k$. Write $L=L_{t_k}$, so $t_k=L/a_k$. Every integer $n$ with $|n-a_k|<B_k/2$ lies in $I_k^-$. Since $nt_k-L=t_k(n-a_k)$, every $n\in S$ therefore satisfies $|nt_k-L|\ge M_k:=t_kB_k/2$. Since $L\le e^L$, we have $a_k=Le^L\le e^{2L}$, so $L\ge\frac12\log a_k$; the same holds with $b_k$ below in place of $a_k$. Since $X_k\le a_k\le2X_k$ and $B_k\ge h(X_k)-1$,
\[
M_k=\frac{LB_k}{2a_k}\ge\frac{B_k\log X_k}{8X_k}\ge\frac{\nu(X_k)}8-\frac{\log X_k}{8X_k}\longrightarrow\infty .
\]
Because $g$ is decreasing on $[0,\infty)$ and increasing on $(-\infty,0]$, comparison with integrals gives
\[
\sum_{n:\ |nt-L|\ge M}t\,g(nt-L)\le\int_{|u|\ge M-t}g(u)\dd u \qquad (M\ge t).
\]
With Lemma~\ref{lem:gumbel},
\[
F_S(e^{-t_k})\le\sum_{n:\ |nt_k-L|\ge M_k}\omega_n(e^{-t_k})\le\int_{|u|\ge M_k-t_k}g+8t_k\longrightarrow0 .
\]
Similarly, let $b_k=X_k+3B_k/2$ and $x_{t'_k}=b_k$. Every integer $n$ with $|n-b_k|<B_k/2$ lies in $I_k^+\subseteq S$, and $b_k\le3X_k$ gives $M'_k:=t'_kB_k/2\ge B_k\log X_k/(12X_k)\to\infty$. The same estimate gives $F_{\N\setminus S}(e^{-t'_k})\to0$. By~\eqref{eq:mass}, $F_S=1-(q;q)_\infty-F_{\N\setminus S}$, so $F_S(e^{-t'_k})\to1$. Since $0\le F_S\le1$, this proves~(b). \qed

\subsection{Proof of Corollary~\ref{cor:ND}}
Apply Theorem~\ref{thm:counter}(b) with $h(x)=x/\sqrt{\log x}$. Then $h(x)\log x/x=\sqrt{\log x}\to\infty$ and $h(x)=o(x)$, so $S$ has natural density $\frac12$ while $F_S(q)$ oscillates between $0$ and $1$. The function $f=\mathbf 1_S$ is bounded with $M(f)=\frac12$, and the radial path lies in every Stolz sector at $1$. \qed

\begin{remark}
A concrete set with the properties of Corollary~\ref{cor:ND} is $S=\{n\ge2:\ \lfloor(\log n)^{3/2}\rfloor\ \text{even}\}$. Its $j$-th block $[e^{j^{2/3}},e^{(j+1)^{2/3}})$ has length $\asymp x(\log x)^{-1/2}$ at height $x$. The block lengths increase with $j$, and the numbers of integers in consecutive blocks differ from these lengths by at most $1$. The alternating sum of these counts therefore gives $A_S(x)=x/2+O(x(\log x)^{-1/2})$, the rounding contributing only $O((\log x)^{3/2})$. On the other hand, taking $x_t$ at the midpoint of a block of length $B$ at height $x$, the block contains all $n$ with $|nt-L_t|<M$, where $M=tB/2\asymp(\log x)^{1/2}\to\infty$. The argument above then applies verbatim.
\end{remark}

\subsection{Proof of Theorem~\ref{thm:primes}}\label{sec:primesproof}
Fix $\alpha\ge0$. For $S\subseteq\N$ let $\mu_S$ be the measure on $[0,\infty)$ given by
\[
\mu_S(\mathcal B)=\sum_{n\in S,\ v(n)\in\mathcal B}\frac{(\log n)^{1+\alpha}}{n},\qquad I_S(w)=\int g(v-w)\dd\mu_S(v)=\sum_{n\in S}\frac{(\log n)^{1+\alpha}}{n}\,g\big(v(n)-w\big);
\]
for Borel sets $\mathcal B\subseteq[0,\infty)$; note that $\mu_S(\{0\})=0$. Consider the condition
\begin{equation*}\tag{B}\label{eq:B}
A_S\Big(x+\frac{x}{\log x}\Big)-A_S(x)\le\frac{Cx}{(\log x)^{1+\alpha}}\qquad(x\ge3)
\end{equation*}
for some constant $C$. Write $k_t(x)=t\,g(tx-L_t)$ as before, and $\varrho(y)=e^y\exp(-e^y)$, so that $g(u)=\varrho(-u)$ for $u\le0$ and $\varrho$ is decreasing on $[0,\infty)$. Put $\bar g_j=\max_{[j,j+1]}g$; then $\sum_{j\in\Z}(2+|j|)\bar g_j<\infty$. The function $(\log y)^{1+\alpha}/y$ is decreasing for $y\ge e^{1+\alpha}$.

\begin{lemma}\label{lem:muwindows}
Let $S$ satisfy~\eqref{eq:B}.
\begin{enumerate}
\item[(a)] $M_S:=\sup_{w\ge0}\mu_S((w,w+1])<\infty$.
\item[(b)] For $\delta\ge0$, condition~\textup{(ii)} of Theorem~\ref{thm:primes} holds if and only if $\mu_S((w,w+c])\to c\delta$ as $w\to\infty$ for every $c>0$.
\end{enumerate}
\end{lemma}

\begin{proof}
Let $w$ be large and $X=\chi(w)$. Then $\chi(w+c)=X\exp(\sqrt{2w+2c}-\sqrt{2w})\le X\exp(c/\log X)$. For $n\in(X,\chi(w+c+1)]$, uniformly, $(\log n)^{1+\alpha}/n=((\log X)^{1+\alpha}/X)(1+o(1))$.

(a) The interval $(X,\chi(w+1)]$ is covered by at most three intervals $(z,z+z/\log z]$ with $X\le z\le2X$. By~\eqref{eq:B} it contains $O(X/(\log X)^{1+\alpha})$ elements of $S$, each of weight at most $(\log X)^{1+\alpha}/X$. Bounded $w$ contribute a finite amount.

(b) Put $Y_c=X+cX/\log X$. By the proof of Lemma~\ref{lem:window}, $v(Y_c)=w+c+o(1)$, hence
\[
\mu_S\big((w,v(Y_c)]\big)=\frac{(\log X)^{1+\alpha}}{X}\big(A_S(Y_c)-A_S(X)\big)(1+o(1)),
\]
and by~\eqref{eq:B} the factor $1+o(1)$ may be replaced by an additive $o(1)$. Let $0<\eta<c$. For large $w$ we have $w+c-\eta\le v(Y_c)\le w+c+\eta$ and $v(Y_{c-\eta})\le w+c\le v(Y_{c+\eta})$. Both the $\mu_S$-mass and the count increase with $c$, so each of the two window conditions places the other quantity between its values at $c-\eta$ and $c+\eta$, up to $o(1)$. Letting $\eta\to0$ proves the equivalence.
\end{proof}

\begin{lemma}\label{lem:primesreduction}
If $S$ satisfies~\eqref{eq:B}, then $L_t^{\alpha}F_S(e^{-t})-I_S(w(t))\to0$ as $t\to0^+$.
\end{lemma}

\begin{proof}
Let $L=L_t\ge8$ and, for $n\in S$, put $u_n=tn-L$. As in the proof of Lemma~\ref{lem:reduction} we have $\log n=L+\log(u_n+L)$, $v(n)-w(t)=\sigma_t(u_n)$ and $\log n/n=t\,\sigma_t'(u_n)$. Put
\[
a_n=L^{\alpha}t\,g(u_n),\qquad b_n=\frac{(\log n)^{1+\alpha}}{n}\,g\big(\sigma_t(u_n)\big)=t(\log n)^{\alpha}\sigma_t'(u_n)\,g\big(\sigma_t(u_n)\big).
\]
By Lemma~\ref{lem:gumbel}, $|L^{\alpha}F_S(e^{-t})-\sum_{n\in S}a_n|\le8tL^{\alpha}\to0$, and $I_S(w(t))=\sum_{n\in S}b_n$.

\emph{Counting.} Let $N_j=\#\{n\in S:\ j\le u_n<j+1\}$ for integers $j\ge j_0:=\lfloor-L/2\rfloor$. These $n$ lie in $[y,y+1/t)$ with $y=(L+j)/t$, $L+j\ge L/2-1\ge3$, and $\log y=L+\log(L+j)$. Since $(L+\log(L+j))/(L+j)\le4$, this interval is covered by at most five intervals $(z,z+z/\log z]$ with $y-1\le z\le2(y+1/t)$. By~\eqref{eq:B}, and since $\log(L+j)\ge0$,
\begin{equation}\label{eq:Nj}
N_j\le\frac{10C(L+j+1)}{t\,(L+\log(L+j))^{1+\alpha}}\le\frac{10C(L+|j|+1)}{tL^{1+\alpha}}\qquad(j\ge j_0).
\end{equation}

\emph{Tails of $a$.} By~\eqref{eq:Nj}, the $n$ with $u_n\in[j,j+1)$, $j\ge j_0$, contribute at most $10C(2+|j|)\bar g_j$ to $\sum a_n$. The remaining $n\in S$ satisfy $u_n<j_0\le-L/2$, so $n<L/(2t)$ and $a_n\le L^{\alpha}t\,\varrho(L/2)$. They contribute at most $\frac12L^{1+\alpha}\varrho(L/2)\to0$.

\emph{Tails of $b$.} For $j\ge j_0$, since $(\log n)^{1+\alpha}/n\le(\log y)^{1+\alpha}/y$ for $n\ge y$, the $n$ with $u_n\in[j,j+1)$ contribute at most
\[
N_j\,\frac{t(L+\log(L+j))^{1+\alpha}}{L+j}\,\bar g'_j\le\frac{10C(L+j+1)}{L+j}\,\bar g'_j\le20C\,\bar g'_j,\qquad \bar g'_j=\sup_{u\in[j,j+1)}g(\sigma_t(u)).
\]
Since $\log(s+L)\ge0$ for $s\ge j_0$, we have $\sigma_t'(s)\ge L/(s+L)$. For $u\ge0$ this gives $\sigma_t(u)\ge L\log(1+u/L)$, so $\bar g'_j\le2^{-j}$ for $0\le j\le L$ and $\bar g'_j\le2^{-L/2}(1+j/L)^{-L/2}$ for $j>L$. The sum of $20C\,\bar g'_j$ over $j>L$ is therefore $O(2^{-L/2}L)\to0$. For $j_0\le u\le0$ we have $\sigma_t'\ge1$ on $[u,0]$, hence $\sigma_t(u)\le u$ and $\bar g'_j\le\varrho(|j+1|)$ for $j\le-1$. For the $n$ with $u_n<j_0$, which satisfy $n<y_1=L/(2t)$, we have
\[
w(t)-v(n)\ge\tfrac12(\log x_t-\log y_1)\log x_t\ge L/3 .
\]
Their contribution is therefore at most $\varrho(L/3)\sum_{n<y_1}(\log n)^{1+\alpha}/n\le\varrho(L/3)(2L)^{2+\alpha}\to0$.

Consequently there is $\eps(U)\to0$ as $U\to\infty$ such that, for all sufficiently small $t$, the $n\in S$ with $|u_n|\ge U$ contribute at most $\eps(U)$ to $\sum a_n$ and to $\sum b_n$.

\emph{Bulk.} For $|u|\le U$ we have $a_n/b_n=L^{\alpha}g(u)/((\log n)^{\alpha}\sigma_t'(u)g(\sigma_t(u)))$ with $u=u_n$. As $t\to0^+$, uniformly for $|u|\le U$: $L/\log n\to1$, $\sigma_t'(u)\to1$ and $\sigma_t(u)\to u$. Since $g$ is continuous and positive on compact sets, $a_n/b_n\to1$ uniformly. By~\eqref{eq:Nj} the sum $\sum_{|u_n|\le U}a_n$ is bounded, so the bulk contributes $o(1)$ to $\sum(a_n-b_n)$. Hence $\limsup_{t\to0^+}|L^{\alpha}F_S(e^{-t})-I_S(w(t))|\le2\eps(U)$ for every $U$.
\end{proof}

The next lemma is a form of the Wiener--Pitt Tauberian theorem for nonnegative measures (cf.~\cite{Korevaar}); we include a short proof based on Theorem~\ref{thm:wiener}.

\begin{lemma}\label{lem:measureTauber}
Let $\mu\ge0$ be a Borel measure on $[0,\infty)$ with $M=\sup_{w\ge0}\mu((w,w+1])+\mu(\{0\})<\infty$, and let $\delta\ge0$. Then $J(w)=\int g(v-w)\dd\mu(v)\to\delta$ as $w\to\infty$ if and only if $\mu((w,w+c])\to c\delta$ for every $c>0$.
\end{lemma}

\begin{proof}
For every $w$, the contribution of $|v-w|\ge U$ to $J(w)$ is at most $2M\sum_{|j|\ge U-1}\bar g_j$, which is small for large $U$.

Suppose first that the windows converge. Fix $U$ and $\eps>0$, and choose step functions $s_\pm$ supported in $[-U,U]$ with $s_-\le g\le s_+$ there and $\int(s_+-s_-)<\eps$. For a step function $s$, $\int s(v-w)\dd\mu(v)$ is a finite combination of window masses, so it tends to $\delta\int s$. Together with the tail bound, this gives $J(w)\to\delta$.

Conversely, suppose $J(w)\to\delta$ and let $0<\eta<1$. Put $\phi(v)=\eta^{-1}\mu((v-\eta,v])$ for $v\in\R$; then $0\le\phi\le M/\eta$. Let $\widetilde\Phi=\phi+\delta\mathbf 1_{(-\infty,0)}$ and $K(z)=g(-z)$. By Fubini,
\[
(K*\widetilde\Phi)(w)=\frac1\eta\int_0^\eta J(w-s)\dd s+\delta\,\mathcal G(-w).
\]
This tends to $\delta$ as $w\to+\infty$. It also tends to $\delta$ as $w\to-\infty$, because $g(z)\le e^{-z}$ gives $J(w')\le M\sum_{j\ge0}e^{-(j+|w'|)}\to0$ as $w'\to-\infty$. By~\eqref{eq:fourier}, $\widehat K(\xi)=\Gamma(1-i\xi)\ne0$, and $\widehat K(0)=1$. Theorem~\ref{thm:wiener} with $k=\mathbf 1_{[-c,0]}$ gives $\int_w^{w+c}\phi\to c\delta$ for every $c>0$. Moreover
\[
\int_w^{w+c}\phi(v)\dd v=\int\varpi(y)\dd\mu(y),\qquad \varpi(y)=\eta^{-1}\big|[w,w+c]\cap[y,y+\eta)\big|\in[0,1],
\]
and $\varpi=1$ on $[w,w+c-\eta]$ while $\varpi=0$ outside $(w-\eta,w+c]$. Since $\mu\ge0$,
\[
\int_{w+\eta}^{w+c}\phi\le\mu((w,w+c])\le\int_w^{w+c+\eta}\phi .
\]
Hence all limit points of $\mu((w,w+c])$ lie in $[(c-\eta)\delta,(c+\eta)\delta]$, and $\eta\to0$ completes the proof.
\end{proof}

\emph{Completion of the proof.} Since $\log(1/(1-q))=L_t+O(t)$ for $q=e^{-t}$ and $0\le F_S\le1$, condition~(i) is equivalent to $L_t^{\alpha}F_S(e^{-t})\to\delta$.

Assume~(i). For small $t$, every $n\in S$ with $|tn-L_t|\le1$ satisfies $k_t(n)\ge t\,g(-1)=te^{1-e}$. By Lemma~\ref{lem:gumbel},
\[
\#\big(S\cap[x_t-1/t,\,x_t+1/t]\big)\le\frac{e^{e-1}(F_S(e^{-t})+8t)}{t}\le\frac{C_0}{tL_t^{\alpha}}\le\frac{2^{1+\alpha}C_0\,x_t}{(\log x_t)^{1+\alpha}},
\]
using $\log x_t=L_t+\log L_t\le2L_t$. Since $x_t/\log x_t\le1/t$ and $t\mapsto x_t$ maps $(0,t_0)$ onto $(x_{t_0},\infty)$ for a suitable small $t_0$, condition~\eqref{eq:B} follows, after enlarging $C$ to cover $3\le x\le x_{t_0}$. Lemma~\ref{lem:primesreduction} gives $I_S(w(t))\to\delta$, hence $I_S(w)\to\delta$ as $w\to\infty$. Lemmas~\ref{lem:muwindows}(a) and~\ref{lem:measureTauber} then give $\mu_S((w,w+c])\to c\delta$ for every $c>0$, and Lemma~\ref{lem:muwindows}(b) gives~(ii).

Conversely, (ii) with $c=1$ implies~\eqref{eq:B}. By Lemma~\ref{lem:muwindows}(b), $\mu_S((w,w+c])\to c\delta$ for every $c>0$. Lemma~\ref{lem:measureTauber} gives $I_S(w)\to\delta$, and Lemma~\ref{lem:primesreduction} gives~(i). \qed

\subsection{Proof of Corollaries~\ref{cor:primesuff} and~\ref{cor:chebotarev}}
Let $H=cx/\log x$. On $[x,x+H]$ we have $(\log y)^{-\alpha}=(\log x)^{-\alpha}(1+O(1/\log x))$, and the derivative of $y/(\log y)^{1+\alpha}$ is $O((\log x)^{-1-\alpha})$. Hence the hypothesis of Corollary~\ref{cor:primesuff} gives
\[
A_S(x+H)-A_S(x)=\frac{\delta H}{(\log x)^{\alpha}}+O\Big(\frac{x}{(\log x)^{2+\alpha}}\Big)+o\Big(\frac{x}{(\log x)^{1+\alpha}}\Big)=\frac{c\delta x}{(\log x)^{1+\alpha}}+o\Big(\frac{x}{(\log x)^{1+\alpha}}\Big),
\]
which is condition~(ii) of Theorem~\ref{thm:primes}.

For Corollary~\ref{cor:chebotarev}, the Chebotarev density theorem with the classical error term~\cite{LagariasOdlyzko} gives $A_{\mathbb P_C}(x)=\frac{|C|}{[K:\mathbb Q]}\mathrm{li}(x)+O_K(x\exp(-c_K\sqrt{\log x}))$ for some $c_K>0$, so Corollary~\ref{cor:primesuff} applies with $\alpha=1$, $\delta=|C|/[K:\mathbb Q]$ and $\delta'=0$. The case of arithmetic progressions is $K=\mathbb Q(e^{2\pi i/k})$, where the primes $p\equiv\ell\ (\mathrm{mod}\ k)$ with $p\nmid k$ are those with Frobenius class $\{\sigma_\ell\}$, where $\sigma_\ell(e^{2\pi i/k})=e^{2\pi i\ell/k}$, and $[K:\mathbb Q]=\varphi(k)$. Finitely many primes do not affect the hypothesis. By Lemma~\ref{lem:identities}, $F_S$ is the displayed partition series. \qed

\subsection{Proof of Lemma~\ref{lem:poch}}
Expanding the logarithm and summing over $k$,
\[
\log(aq;q)_\infty=-\sum_{r\ge1}\frac{a^r}{r}\,\frac{q^r}{1-q^r}=-\frac{\mathrm{Li}_2(a)}{\tau}-\sum_{r\ge1}\frac{a^r}{r}\,H(r\tau),\qquad H(z)=\frac1{e^z-1}-\frac1z .
\]
The function $H$ is holomorphic in $|\mathrm{Im}\,z|<2\pi$ with $H(0)=-\frac12$, and its only poles are at $2\pi i\Z\setminus\{0\}$. On the ray $\{s(1+i\beta):s\ge0\}$, which avoids these poles, we have $|e^z|=e^s$. Hence $|H(z)|\le C_\beta\min(1,|z|^{-1})$ and $|H'(z)|\le C_\beta\min(1,|z|^{-2})$ there. For~(a), the last sum is at most $C_\beta\sum_{r\ge1}r^{-1}\min(1,(r|\tau|)^{-1})\le C_\beta(3+\log(1/t))$.

For~(b), let $T_r=\sum_{k\ge r}a^k/k$, which converges and tends to $0$ because $a\ne1$. Put $u_r=H(r\tau)$. Summation by parts gives
\[
\sum_{r\ge1}\frac{a^r}ru_r=T_1u_1+\sum_{r\ge1}T_{r+1}(u_{r+1}-u_r).
\]
The total variation $\sum_r|u_{r+1}-u_r|$ is at most $\int_0^\infty|H'(s\tau)||\tau|\dd s\le V_\beta<\infty$, uniformly in $t$. Each difference $u_{r+1}-u_r$ tends to $0$, and $u_1\to H(0)$. Splitting the last sum at a large $R$ therefore shows that it tends to $0$. Hence $\sum_r a^ru_r/r\to-\frac12T_1=\frac12\mathrm{Log}(1-a)$, and
\[
(aq;q)_\infty=(1-a)^{-1/2}\exp(-\mathrm{Li}_2(a)/\tau)(1+o(1)).
\]
By~\eqref{eq:eta}, since $|e^{-4\pi^2/\tau}|=\exp(-4\pi^2/(t(1+\beta^2)))\to0$, we have
\[
(q;q)_\infty=(2\pi/\tau)^{1/2}\exp(-\pi^2/(6\tau))(1+o(1)).
\]
Dividing proves the first formula in~(b). For the second, by~\eqref{eq:Li2circle} and~\eqref{eq:kappa},
\[
-\frac{\pi^2/6-\mathrm{Li}_2(e^{i\vartheta})}{\tau}=\frac{-(\gamma(\vartheta)-i\,\mathrm{Cl}_2(\vartheta))(1-i\beta)}{t(1+\beta^2)}=\frac{\kappa_\beta(\vartheta)}{t}+i\,\frac{\psi_\beta(\vartheta)}{t},
\]
which gives the modulus. \qed

\subsection{Proof of Theorem~\ref{thm:periodic}}
For $|q|<1$ and $|a|\le1$, Euler's identity with $z=aq$ gives $\sum_{n\ge0}a^nq^n(q;q)_\infty/(q;q)_n=(q;q)_\infty/(aq;q)_\infty$, that is,
\begin{equation}\label{eq:euler-weights}
\sum_{n\ge1}a^n\omega_n(q)=\frac{(q;q)_\infty}{(aq;q)_\infty}-(q;q)_\infty .
\end{equation}
Write $\zeta=e^{2\pi i/m}$ and $f(n)=\sum_{j=0}^{m-1}\widehat f(j)\zeta^{jn}$. By~\eqref{eq:euler-weights},
\[
\cA_f(q)=\widehat f(0)\big(1-(q;q)_\infty\big)+\sum_{j=1}^{m-1}\widehat f(j)\big(R_j(q)-(q;q)_\infty\big),\qquad R_j(q)=\frac{(q;q)_\infty}{(\zeta^jq;q)_\infty}.
\]
By~\eqref{eq:eta}, $(q_t;q_t)_\infty\to0$. By Lemma~\ref{lem:poch}(b) with $\vartheta_j=2\pi j/m$,
\[
R_j(q_t)=c_j\,t^{-1/2}\exp\Big(\frac{\kappa_\beta(\vartheta_j)+i\psi_\beta(\vartheta_j)}{t}\Big)(1+o(1)),\qquad c_j\ne0 .
\]
If $\kappa^*<0$, every term with $\widehat f(j)\ne0$ tends to $0$, which proves~(a).

Let $\kappa^*\ge0$ and let $\mathcal J$ be the set of $j$ with $\widehat f(j)\ne0$ and $\kappa_\beta(\vartheta_j)=\kappa^*$. All other terms, including the bounded term with $j=0$, are $o(t^{-1/2}e^{\kappa^*/t})$. Hence
\[
t^{1/2}e^{-\kappa^*/t}\cA_f(q_t)=\sum_{j\in\mathcal J}\widehat f(j)c_j\,e^{i\psi_\beta(\vartheta_j)/t}+o(1).
\]
The frequencies $\psi_\beta(\vartheta_j)$, $j\in\mathcal J$, are distinct. Indeed, equal values of $\kappa_\beta$ and $\psi_\beta$ at $\vartheta_j$ and $\vartheta_k$ give $(1+\beta^2)(\gamma(\vartheta_j)-\gamma(\vartheta_k))=0$ and then $\mathrm{Cl}_2(\vartheta_j)=\mathrm{Cl}_2(\vartheta_k)$. Since $\gamma$ is injective on $[0,\pi]$ and symmetric about $\pi$, this forces $\vartheta_k\in\{\vartheta_j,2\pi-\vartheta_j\}$, and $\mathrm{Cl}_2(2\pi-\vartheta)=-\mathrm{Cl}_2(\vartheta)$ with~\eqref{eq:clausen} gives $\vartheta_k=\vartheta_j$. For a trigonometric sum $P(s)=\sum_jb_je^{i\psi_js}$ with distinct real $\psi_j$, the mean of $|P(s)|^2$ over $s\in[s_0,s_0+X]$ tends to $\sum_j|b_j|^2$ as $X\to\infty$. Hence $\limsup_{s\to\infty}|P(s)|^2\ge\sum_j|b_j|^2>0$, and~(b) follows with $s=1/t$. \qed

\subsection{Proof of Corollary~\ref{cor:AP}}
For $f=\mathbf 1_{S_{r,m}}$ one has $\widehat f(j)=m^{-1}e^{-2\pi ijr/m}\ne0$ for all $j$. Since $\gamma(2\pi-\vartheta)=\gamma(\vartheta)$ and $\mathrm{Cl}_2(2\pi-\vartheta)=-\mathrm{Cl}_2(\vartheta)$, the pairs $j$ and $m-j$ contribute $(\pm\beta\,\mathrm{Cl}_2(\vartheta_j)-\gamma(\vartheta_j))/(1+\beta^2)$. The index $j=m/2$, if present, contributes $-\gamma(\pi)/(1+\beta^2)<0$. By~\eqref{eq:clausen}, $\kappa^*<0$ if and only if $|\beta|\,\mathrm{Cl}_2(\vartheta_j)<\gamma(\vartheta_j)$ for all $1\le j<m/2$, that is, $|\beta|<\beta_m$, and Theorem~\ref{thm:periodic} applies. For $m=1$ there is no nonzero frequency, and for $m=2$ the only one is $\vartheta=\pi$, where $\mathrm{Cl}_2(\pi)=0$. Since $2\sin(x/2)\le x$,~\eqref{eq:clausen} gives $\mathrm{Cl}_2(\vartheta)\ge\vartheta(1+\log(1/\vartheta))$, while $\gamma(\vartheta)\le\pi\vartheta/2$; with $\vartheta=2\pi/m$ this gives the bound for $\beta_m$. For $m=3$ and $m=4$ only $j=1$ occurs, with $\gamma(2\pi/3)=2\pi^2/9$, $\gamma(\pi/2)=3\pi^2/16$ and $\mathrm{Cl}_2(\pi/2)=G$. For $m=6$ we have $\gamma(\pi/3)=5\pi^2/36$, and the value $1.3506\ldots$ at $j=1$ is smaller than the value $\beta_3=3.2414\ldots$ at $j=2$. \qed

\subsection{Proof of Theorem~\ref{thm:instability}}
Since $\omega_n(\overline q)=\overline{\omega_n(q)}$ and $q_t(-\beta)=\overline{q_t(\beta)}$, we may assume $\beta>0$.

(a) Write $z=e^{-s(1+i\beta)}$, so that $|\mathrm{Li}_2(z)-z|\le|z|^2$. Since $1-i\beta=(1+\beta^2)^{1/2}e^{-i\arctan\beta}$,
\[
\Lambda_\beta(s)\ge-\frac{e^{-s}\cos(\beta s+\arctan\beta)}{(1+\beta^2)^{1/2}}-\frac{e^{-2s}}{(1+\beta^2)^{1/2}},
\]
and $s=s_\beta$ makes the cosine equal to $-1$. Thus $\Psi_\beta>0$. Also $\Psi_\beta\le\pi^2/(6(1+\beta^2)^{1/2})$, $\Lambda_\beta(0)<0$ and $\Lambda_\beta(s)\to0$ as $s\to\infty$, so the supremum is attained at some $s^*>0$. Apply Lemma~\ref{lem:poch}(a) with $a=q^n$, noting $(q^{n+1};q)_\infty=(aq;q)_\infty$ and $\mathrm{Li}_2(q_t^n)/\tau=\mathrm{Li}_2(e^{-nt(1+i\beta)})/(t(1+i\beta))$. This gives, uniformly in $n\ge1$,
\begin{equation}\label{eq:logomega}
\log|\omega_n(q_t)|=-nt+\frac{\Lambda_\beta(nt)}{t}+O_\beta(\log(2/t)).
\end{equation}
The upper bound in~(a) follows since $-nt\le 0$ and $\Lambda_\beta\le\Psi_\beta$. For the lower bound take $n=\lceil s^*/t\rceil$. Then $|nt-s^*|\le t$, and $\Lambda_\beta$ is Lipschitz near $s^*$, so $\Lambda_\beta(nt)\ge\Psi_\beta-O(t)$.

(b) Let $\eps_k=+1$ if $2k\in S$ and $\eps_k=-1$ otherwise. Every block $\{2k-1,2k\}$ contains exactly one element of $S$, so $|A_S(x)-x/2|\le1$, and $F_S(q)\to\frac12$ as $q\to1^-$ by Corollary~\ref{cor:consequences}(a). By~\eqref{eq:mass},
\[
F_S(q)=\tfrac12\big(1-(q;q)_\infty\big)+\sum_{k\ge1}\eps_k\Delta_k(q),\qquad \Delta_k(q)=\tfrac12\big(\omega_{2k}(q)-\omega_{2k-1}(q)\big).
\]
From $\omega_{n-1}=q^{n-1}(1-q^n)(q^{n+1};q)_\infty$ we get $\omega_n-\omega_{n-1}=\omega_n\,(q-1+q^n)/q$. Moreover $|q_t-1|\le|\tau|$ and $|q_t^n|=e^{-nt}$. Fix $0<\eta<\Psi_\beta/2$ and $0<r_0<s^*$ with $\Lambda_\beta\ge\Psi_\beta-\eta$ on $[s^*-r_0,s^*+r_0]$. By~\eqref{eq:logomega}, for all $n$ with $|nt-s^*|\le r_0$ and all small $t$,
\[
|\omega_n(q_t)-\omega_{n-1}(q_t)|\ge\tfrac12e^{-s^*-r_0}|\omega_n(q_t)|\ge\exp\Big(\frac{\Psi_\beta-2\eta}{t}\Big).
\]
Let $M>0$ and $R>(M+1)/\cos(\pi/8)$. Let $V_t$ be the set of $k$ with $|\Delta_k(q_t)|\ge R$; then $|V_t|\ge r_0/(2t)-1$ for small $t$. Each $z\in\C$ with $|z|\ge R$ satisfies $|\mathrm{Re}(e^{-i\phi}z)|\ge R\cos(\pi/8)$ for some $\phi\in\{0,\pi/4,\pi/2,3\pi/4\}$. Choose $\phi$ serving a set $V'_t\subseteq V_t$ with $|V'_t|\ge|V_t|/4$. If $|F_S(q_t)|\le M$, then $\mathrm{Re}\big(e^{-i\phi}\sum_{k\in V'_t}\eps_k\Delta_k\big)$ lies in an interval of length $2(M+1)$, whose position is determined by the $\eps_k$ with $k\notin V'_t$. Conditioning on these and applying the Erd\H{o}s--Littlewood--Offord inequality to the numbers $\mathrm{Re}(e^{-i\phi}\Delta_k)/(R\cos(\pi/8))$, $k\in V'_t$, gives
\[
\Pr\big(|F_S(q_t)|\le M\big)\le|V'_t|^{-1/2}\longrightarrow0\qquad(t\to0^+).
\]
Since $t\mapsto F_S(q_t)$ is continuous, the event $\{\sup_{0<t<t_0}|F_S(q_t)|\le M\}$ is measurable. Its probability is at most $\Pr(|F_S(q_t)|\le M)$ for every $t<t_0$, hence $0$. Taking the union over $M,t_0^{-1}\in\N$ and over $\beta$ in the prescribed countable set proves~(b). \qed

\subsection{Proof of Corollary~\ref{cor:stolz}}
Given $M_0>1$, choose $\beta\ne0$ with $(1+\beta^2)^{1/2}<M_0$. The path $q_t(\beta)$ lies eventually in the sector. By Corollary~\ref{cor:AP}, every $m\ge3$ with $\pi/(2(1+\log(m/2\pi)))\le|\beta|$ gives residue classes with unbounded $F_{S_{r,m}}(q_t)$, and Theorem~\ref{thm:instability}(b) gives the random sets. Periodic functions and sets of bounded discrepancy have a mean value. Lemma~2.5 of~\cite{OSW1} asserts that $(q;q)_\infty/(\zeta^aq;q)_\infty\to0$ as $q\to1$ from within the unit disk whenever $\zeta$ is a primitive $m$-th root of unity and $m\nmid a$; for $\zeta=e^{2\pi i/m}$, Lemma~\ref{lem:poch}(b) shows that this quotient is unbounded along $q_t(\beta)$ whenever $\kappa_\beta(2\pi a/m)\ge0$. \qed

\section{Conclusion}

On the real axis and for bounded $f$, $\lim_{q\to1^-}\cA_f(q)=L$ holds exactly when $f$ has mean $L$ on all intervals $(x,x+cx/\log x]$, because the generating function averages $f$ against the Gumbel law of the largest part. With the normalization $(\log(1/(1-q)))^{\alpha}$, the same mechanism, now with positivity as the Tauberian condition, characterizes sets of local density $(\log x)^{-\alpha}$, and in particular the prime analogue of Alladi's formula. Off the real axis the largest weights are exponentially large, and the formula survives only through cancellation. For periodic functions this cancellation is governed by the Clausen function, and for random sets of bounded discrepancy it fails almost surely. The following problems remain open.

\begin{enumerate}
\item \emph{Non-radial limits.} For given $\beta\ne0$, characterize the bounded $f$ for which $\cA_f(q_t(\beta))$ converges. Theorem~\ref{thm:periodic} settles the periodic case, and Theorem~\ref{thm:instability} shows that bounded discrepancy is far from sufficient. Is there a $\{0,1\}$-valued $f$, not eventually periodic, for which $\cA_f(q_t(\beta))$ converges for some $\beta\ne0$?
\item \emph{The optimal constant.} Let $C^*$ be the least constant such that $\limsup_{q\to1^-}|\cA_f(q)-d|\le C^*\limsup_{x\to\infty}|E(x)|\log x/x$ for all bounded $f$ and all $d\in\C$, with $E$ as in Theorem~\ref{thm:abelian}. Theorem~\ref{thm:abelian} gives $C^*\le2/e\approx0.736$. By~\eqref{eq:AS-osc} and~\eqref{eq:FS-osc}, $C^*\ge\sup_{\theta>0}\theta\,(\pi\theta/\sinh\pi\theta)^{1/2}\approx0.523$. Determine $C^*$.
\item \emph{Roots of unity and other statistics.} In~\cite{OSW1} limits are also taken as $q$ tends to roots of unity $\zeta$. Characterize the functions $f$ for which $\cA_f(q)\to L$ as $q\to\zeta$ radially. More generally, determine the window scale and limiting kernel for density formulas built from other partition statistics, and decide in which cases the kernel's Fourier transform vanishes somewhere, which would make the corresponding Tauberian equivalence fail.
\end{enumerate}


\end{document}